\documentclass[12pt]{amsart}
\usepackage[utf8]{inputenc}

\usepackage{amsmath, amsfonts,amssymb,amsthm,mathtools}
\usepackage{comment}
\usepackage{biblatex}
\usepackage{hyperref}
\usepackage{mathrsfs}

\usepackage[margin=1in]{geometry}
\usepackage{asymptote}

\usepackage[enableskew]{youngtab}
\usepackage{ytableau}
\usepackage{tikz-cd}

\newtheorem{thm}{Theorem}
\newtheorem{lemma}[thm]{Lemma}
\newtheorem{prop}[thm]{Proposition}
\newtheorem{corollary}[thm]{Corollary}

\numberwithin{thm}{section}
\numberwithin{equation}{section}
\numberwithin{figure}{section}

\newtheorem{skewthm}[thm]{Theorem \ref{thm:skew-formula}}

\theoremstyle{definition}

\newtheorem{example}[thm]{Example}
\newtheorem{definition}[thm]{Definition}
\newtheorem{q}[thm]{Question}
\newtheorem{conj}[thm]{Conjecture}
\newtheorem{remark}[thm]{Remark}

\newcommand{\SSYT}{\mathrm{SSYT}}
\newcommand{\cc}{\mathrm{cc}}
\newcommand{\ch}{\mathrm{ch}}
\newcommand{\sh}{\mathrm{sh}}

\newcommand{\revq}{\mathrm{rev}_q}
\newcommand{\sort}{\mathrm{sort}}
\newcommand{\coinv}{\mathrm{coinv}}

\newcommand{\Frob}{\mathrm{Frob}}
\newcommand{\grFrob}{\mathrm{grFrob}}

\newcommand{\la}{\lambda}
\newcommand{\bC}{\mathbb{C}}
\newcommand{\bZ}{\mathbb{Z}}

\newcommand{\bQ}{\mathbb{Q}}
\newcommand{\cO}{\mathcal{O}}
\newcommand{\cN}{\mathcal{N}}
\newcommand{\sP}{\mathscr{P}}
\newcommand{\sB}{\mathscr{B}}
\newcommand{\fp}{\mathfrak{p}}
\newcommand{\fn}{\mathfrak{n}}
\newcommand{\fl}{\mathfrak{l}}
\newcommand{\Lie}{\mathrm{Lie}}
\newcommand{\JT}{\mathrm{JT}}

\newcommand{\Ad}{\mathrm{Ad}}

\usepackage{geometry}

\usepackage{todonotes}

\title[Cocharge and skewing for the $\Delta$-Springer modules]{Cocharge and skewing formulas for \\ $\Delta$-Springer modules and the Delta Conjecture}
\author{Maria Gillespie}
\address{Maria Gillespie, Department of Mathematics, Colorado State University, Fort Collins, CO 80523}
\email{\href{mailto:Maria.Gillespie@colostate.edu}{Maria.Gillespie@colostate.edu}}
\thanks{The first author was partially supported by NSF DMS award number 2054391.}

\author{Sean T.\ Griffin}
\address{Sean T.\ Griffin, Department of Mathematics, University of California Davis, Davis, CA 95616}
\email{\href{mailto:stgriffin@ucdavis.edu}{stgriffin@ucdavis.edu}}

\date{\today}

\bibliography{refs}

\begin{document}

\maketitle

\begin{abstract}
We prove that $\omega \Delta'_{e_{k}}e_n|_{t=0}$, the symmetric function in the Delta Conjecture at $t=0$, is a skewing operator applied to a Hall-Littlewood polynomial, and generalize this formula to the Frobenius series of all $\Delta$-Springer modules. We use this to give an explicit Schur expansion in terms of the Lascoux-Sch\"{u}tzenberger cocharge statistic on a new combinatorial object that we call a \textit{battery-powered tableau}. Our proof is geometric, and shows that the $\Delta$-Springer varieties of Levinson, Woo, and the second author are generalized Springer fibers coming from the partial resolutions of the nilpotent cone due to Borho and MacPherson.

We also give alternative combinatorial proofs of our Schur expansion for several special cases, and give conjectural skewing formulas for the $t$ and $t^2$ coefficients of $\omega \Delta'_{e_{k}}e_n$.

\end{abstract}

\section{Introduction and Main Results}

\textit{Schur positivity} is a central focus of algebraic combinatorics.  One famous example is the Macdonald Positivity Conjecture, proven by Haiman \cite{Haiman}, which states that the symmetric \textit{Macdonald polynomials} $\widetilde{H}_\mu(x;q,t)$ expand in the Schur basis with positive coefficients in $\mathbb{Z}_+[q,t]$. 
 The proof uses the geometry of the Hilbert scheme $\mathrm{Hilb}_{n}(\bC^2)$ of arrangements of $n$ points in the plane $\bC^2$, and no direct combinatorial proof or explicit formula is yet known.

The \textit{Delta Conjecture} \cite{HRW}, which generalizes the recently-proven Shuffle Theorem \cite{CarlssonMellit}, motivates a major current area of research in symmetric function theory (e.g. \cite{BHMPS}, \cite{HRS},\cite{Schedules},  \cite{RhoadesPawlowski}).  It states two combinatorial formulas, in terms of parking functions, for $\Delta'_{e_{k-1}}e_n$ where $\Delta'_f$ is a particular eigenoperator of the Macdonald polynomials defined for any symmetric function $f$.   One of the two Delta Conjecture formulas has been proven in \cite{BHMPS,DAdderio-Mellit}. 

The Shuffle Theorem concerns the special case when $k=n$, in which $\Delta'_{e_{n-1}} e_n$ is the bi-graded \textit{Frobenius series} (in $q,t$) of the \textit{diagonal coinvariant ring} $$\mathrm{DR}_n=\mathbb{Q}[x_1,\ldots,x_n,y_1,\ldots,y_n]/I_n$$ where $I_n$ is generated by the $S_n$-invariants with no constant term under the diagonal action of $S_n$ permuting the $x$'s and $y$'s simultaneously.    While the Shuffle Theorem gives a monomial expansion for $\Delta'_{e_{n-1}}e_n$, an explicit formula for the Schur expansion is not known (and similarly for the Delta Conjecture).

In particular, the decomposition of a graded $S_n$-module $R=\bigoplus_d R_d$ into irreducibles can be described by its \textit{graded Frobenius character}
$$\grFrob(R):=\sum_d \Frob(R_d)q^d$$ where $R_d$ is the $d$-th graded piece and $\Frob$ is the additive map on representations that sends the irreducible $S_n$-module $V_\nu$ to the Schur function $s_\nu$.  For a bi-graded module, we use two parameters $q,t$ and obtain a bi-variate generating series.  This means that determining the Schur expansion for Macdonald polynomials, the Shuffle theorem polynomials, or those of the Delta Conjecture would lead to a deeper understanding of the $S_n$-representation theory of the associated (bi-)graded modules.

In the one-parameter case, setting $t=0$ often leads to more tractable problems.  For instance, a famous result of Lascoux and Sch\"utzenberger was their discovery of the \textit{cocharge} statistic on Young tableaux to give a combinatorial formula for the Schur expansion of the (modified) Hall-Littlewood polynomials $\widetilde{H}_{\mu}(x;q)$, which are the $t=0$ specialization of the Macdonald polynomials. The polynomials $\widetilde{H}_{\mu}(x;q)$ are the graded Frobenius character of the \emph{Garsia-Procesi} modules $R_\mu$.  These $S_n$-modules in turn are the cohomology rings of \textit{Springer fibers} $\sB_\mu$.  The cocharge statistic therefore resolved the natural question of how $R_\mu$ decomposes into irreducible $S_n$-modules.


In particular, for a partition $\mu$, define $\SSYT(\mu)$ to be the set of all (straight shape) semistandard Young tableaux of content $\mu$, meaning that the tableau entries consist of $\mu_i$ copies of $i$ for each $i$, and the entries are weakly increasing across rows and strictly increasing up columns in French notation (as in the ``device'' part of the tableau at left in Figure \ref{fig:battery}).   Lascoux and Sch\"utzenberger showed that \begin{equation}\label{eq:lascoux-schutzenberger}\grFrob(R_\mu)=\widetilde{H}_\mu(x;q)=\sum_{T\in \SSYT(\mu)} q^{\cc(T)} s_{\sh(T)}=\sum_{\nu} \widetilde{K}_{\nu,\mu}(q)s_\nu\end{equation} where $\sh(T)$ is the \textbf{shape} of the tableau $T$, that is, the partition whose $i$-th part is the length of the $i$-th row of $T$ from the bottom, and $s_{\sh(T)}$ is the corresponding Schur function.  Above, $\widetilde{K}_{\nu,\mu}(q)$ is the \textit{$q$-Kostka polynomial}, and $\cc$ is the \textit{cocharge} statistic as defined in Section \ref{sec:background}.

One of the main results of this article generalizes the Lascoux--Sch\"utzenberger formula to the cohomology rings of the \emph{$\Delta$-Springer varieties}, which were recently introduced by Levinson, Woo and the second author~\cite{GLW}.  These graded $S_n$-modules are denoted by $R_{n,\lambda,s}$ and simultaneously generalize both the Garsia-Procesi modules $R_\mu$ and the generalized coinvariant rings $R_{n,k}$ that were defined by Haglund, Rhoades, and Shimozono \cite{HRS} to give an algebraic realization of the Delta Conjecture polynomial $\Delta'_{e_{k-1}}e_n$ at $t=0$.  We obtain this result by connecting the $\Delta$-Springer varieties to the theory of partial resolutions of nilpotent varieties due to Borho and MacPherson \cite{Borho-MacPherson}.

The rings $R_{n,\lambda,s}$, first introduced in \cite{Griffin-Thesis}, are defined for integers $n,s$ and a partition $\lambda$ with $|\lambda|=k\le n$ and $s\ge \ell(\lambda)$. In the special case when $n=|\mu|$, the ring $R_{n,\mu,s}$ coincides with $R_\mu$. When $\lambda = (1^k)$ and $s=k$, the ring $R_{n,\lambda,s}$ coincides with $R_{n,k}$.  Because the common generalization $R_{n,\lambda,s}$ has a geometric interpretation as the cohomology rings of the $\Delta$-Springer varieties 
 $Y_{n,\la,s}$ \cite{GLW}, we refer to them here as the \textit{$\Delta$-Springer modules}.

\subsection{New skewing, charge and cocharge formulas for \texorpdfstring{$R_{n,\la,s}$}{}}

We prove that the graded Frobenius character $\widetilde{H}_{n,\la,s}:=\grFrob(R_{n,\la,s})$ has the following skewing formula.

\begin{thm}\label{thm:skew-formula}
Let $\Lambda = ((n-k)^s) + \lambda$, where addition is computed coordinate-wise. We have $$\widetilde{H}_{n,\lambda,s}(x;q) = \frac{s_{((n-k)^{s-1})}^\perp \widetilde{H}_\Lambda(x;q)}{q^{\binom{s-1}{2}(n-k)}}.$$  
\end{thm}

In the above statement, $s_\nu^\perp$ denotes the adjoint operator to multiplication by $s_\nu$ with respect to the Hall inner product on symmetric functions.

The proof of Theorem \ref{thm:skew-formula} relies heavily on the work of Borho and MacPherson on partial resolutions of the nilpotent cone.  We show that the $\Delta$-Springer varieties $Y_{n,\la,s}$ are instances of the family of varieties studied in their work \cite{Borho-MacPherson}.  We prove a rational smoothness condition that enables us to use a result in \cite{Borho-MacPherson} derived using the theory of perverse sheaves to obtain the Frobenius character.

As an immediate corollary, we have the following simple formula for the symmetric function in Delta Conjecture at $t=0$. We write $\revq$ for the operation of reversing the coefficients of the $q$ polynomial, by setting $q\to q^{-1}$ and multiplying by $q^d$ where $d$ is the degree. We also write $H_\mu(x;q) = \revq(\widetilde{H}_\mu(x;q))$ for the (transformed) Hall-Littlewood symmetric functions.

\begin{corollary}\label{cor:DeltaPerpFormula}
    In the $R_{n,k}$ case, we have
    \[
    \grFrob(R_{n,k}) = \omega\circ \revq(\Delta'_{e_{k-1}}e_n|_{t=0}) = \frac{s_{(n-k)^{k-1}}^\perp \widetilde{H}_{((n-k+1)^k)}(x;q)}{q^{\binom{k-1}{2}(n-k)}}.
    \]
    Equivalently,
    \[
\omega\Delta'_{e_{k-1}}e_n|_{t=0} = s_{((n-k)^{k-1})}^\perp H_{((n-k+1)^k)}(x;q).
    \]
\end{corollary}

We now provide a combinatorial Schur expansion for $\grFrob(R_{n,\la,s})$ that generalizes Equation \eqref{eq:lascoux-schutzenberger}.  We first make more rigorous the definition of the partition $\Lambda$ mentioned above.

\begin{definition}\label{def:Lambda}
For a fixed $n,\lambda,s$ with $k=|\lambda|\le s$, define $\Lambda_{n,\la,s}$ to be the partition formed by adding an $s\times (n-k)$ rectangle at the left of the diagram of $\lambda$.  In other words  $\Lambda_{n,\la,s}=(n-k+\lambda_1,n-k+\lambda_2,\ldots,n-k+\lambda_r,n-k,\ldots,n-k)$ where there are $s$ parts in total.  As an example, for $n=5$, $\la=(2,1)$, $s=4$, we have  $\Lambda_{n,\la,s}=(5,4,3,3)$.
\end{definition}

\begin{definition}
A \textbf{battery-powered tableau} of parameters $n,\la,s$ consists of a pair $T=(D,B)$ of semistandard Young tableaux, where $B$ is rectangular of shape $(s-1)\times (n-k)$, and the total content of $D$ and $B$ is $\Lambda_{n,\la,s}$.  We call $D$ the \textbf{device} of $T$ and $B$ the \textbf{battery}.  We define the \textbf{shape} of $T$ to be the shape of its device, that is, $\sh^+(T)=\sh(D)$.

We write $\mathcal{T}^+(n,\la,s)$ to denote the set of all battery-powered tableaux of parameters $n,\la,s$.  For $T\in \mathcal{T}^+(n,\la,s)$, we write $\cc(T)$ and $\ch(T)$, respectively to denote the cocharge and charge of the word formed by concatenating the reading words of $D$ and $B$ in that order (see Section \ref{sec:background}).
\end{definition}

\begin{remark}
We will usually draw the battery down-and-right from the device, as in Figure \ref{fig:battery}, so that the device and the battery together form a \textbf{skew tableau} (that is, a tableau of shape $\theta/\rho$, where $\theta/\rho$ is formed by deleting the diagram of a partition $\rho$ from a larger partition $\theta$).  We write this tableau as $T=(D, B)$.
\end{remark}
We prove the following formula for the graded Frobenius character of $R_{n,\lambda,s}$, which was originally conjectured in \cite{GG}.

\begin{thm}\label{conj:main}
We have $$\widetilde{H}_{n,\la,s}(x;q):=\grFrob(R_{n,\la,s})=\frac{1}{q^{\binom{s-1}{2}(n-k)}}\sum_{T\in \mathcal{T}^+(n,\la,s)} q^{\cc(T)}s_{\sh^+(T)}(x).$$ 
\end{thm}

\begin{figure}
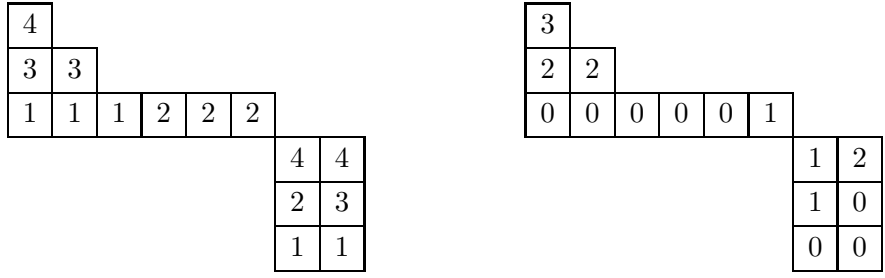

    \centering
\small \begin{ytableau}
 4 \\
 3 & 3  \\ 
 1 & 1 & 1 & 2 & 2 & 2 \\
  \none &\none & \none& \none& \none& \none & 4 & 4 \\
  \none &\none & \none& \none& \none& \none & 2 & 3 \\
  \none &\none & \none& \none& \none& \none &1 & 1
\end{ytableau}\hspace{2cm}
\begin{ytableau}
 3 \\
 2 & 2  \\ 
 0 & 0 & 0 & 0 & 0 & 1 \\
  \none &\none & \none& \none& \none& \none & 1 & 2 \\
  \none &\none & \none& \none& \none& \none & 1 & 0 \\
  \none &\none & \none& \none& \none& \none & 0 & 0
\end{ytableau}
    \caption{At left, a battery-powered tableau $T$ for $n=9$, $\lambda=(3,2,1,1)$, and $s=4$, consisting of a device of shape $(6,2,1)$ and a rectangular battery to its lower right.  The cocharge labels are shown at right, giving $\cc(T)=12$.}
    \label{fig:battery}
\end{figure}

We think of the battery as storing extra charge for the device.  The $q$-exponent $\binom{s-1}{2}(n-k)$  is the largest amount of cocharge that may be stored in the battery.

\begin{example}
Suppose $n=9$, $\la=(3,2,1,1)$, and $s=4$.  Then $\Lambda_{n,\la,s}=(5,4,3,3)$ and an example of a battery-powered tableau is shown in Figure \ref{fig:battery}.  Its cocharge is $12$ and shape is $(6,2,1)$, and the normalization factor in Theorem \ref{conj:main} is $q^{-\binom{3}{2}\cdot 2}=q^{-6}$, so one of the terms of the summation above is $q^{-6}\cdot q^{12}s_{(6,2,1)}=q^6s_{(6,2,1)}$.
\end{example}

In order to prove Theorem \ref{conj:main} from Theorem \ref{thm:skew-formula}, we apply the operator $s_{((n-k)^{s-1})}^\perp$ directly to Equation \eqref{eq:lascoux-schutzenberger}, and in the process, we also obtain the following formula (in the Delta Conjecture case) in terms of Littlewood-Richardson coefficients and $q$-Kostka polynomials.

\begin{corollary}\label{cor:LR-Kostka}
    We have
    \[
    \langle s_\mu, \,\omega\Delta'_{e_{k-1}}e_n|_{t=0}) = \sum_{\nu\vdash k(n-k+1)} c_{\mu,((n-k)^{k-1})}^\nu K_{\nu,((n-k+1)^k)}(q).
    \]
\end{corollary}

By applying $\revq$ to Theorem \ref{conj:main}, we can obtain the following alternative simpler expansion in terms of the generalized charge statistic.

\begin{thm}\label{conj:charge}
  We have $$\revq\left(\widetilde{H}_{n,\la,s}\right)=\revq\left(\grFrob(R_{n,\la,s})\right)=\sum_{T\in \mathcal{T}^+(n,\la,s)} q^{\ch(T)}s_{\sh^+(T)}(x).$$ 
\end{thm}


Specializing to the case relevant to the Delta Conjecture, $\la=(1^k)$ and $s=k$, we have a new Schur expansion for the expression in the Delta Conjecture at $t=0$.

\begin{corollary}[of Theorem~\ref{conj:charge}]\label{cor:Rnk}
    We have 
    \[
    \Delta'_{e_{k-1}}e_n|_{t=0} = \sum_{T\in \mathcal{T}^+(n,(1^k),k)} q^{\ch(T)}s_{\sh^+(T)^*}(x),
    \]
    where $\sh^+(T)^*$ is the transpose of the partition $\sh^+(T)$.
\end{corollary}


Since the proof of Theorems \ref{conj:main} and \ref{conj:charge} that we present here is essentially geometric in nature, it is also of interest to find a more direct combinatorial proof, using the existing expansions of $\grFrob(R_{n,\la,s})$ in terms of monomials or sums of Hall-Littlewood polynomials.  The following theorem summarizes some of our progress towards a combinatorial proof.

\begin{prop}\label{thm:results}
There is a direct combinatorial proof of Theorem \ref{conj:main} for:
\begin{itemize}
    \item $s=2$ and any $n,\la$ (see Section \ref{sec:s=2}),
    \item The coefficient of $s_{(n)}$ in the $t=0$ Delta conjecture case (see Section \ref{sec:Rnk}).
\end{itemize}
\end{prop}

This proposition was stated without full proof details in the conference proceedings article \cite{GG}, and we provide the complete proofs in this paper.   In the companion paper \cite{preprint} to this work, the authors will provide combinatorial proofs of two additional special cases using a new formula in terms of creation operators and the Loehr-Warrington algorithms on \textit{abaci}.

\subsection{Outline}

After establishing background definitions and notation in Section \ref{sec:background}, we prove Theorem~\ref{thm:skew-formula} in Section \ref{sec:MainThm}. We then prove Theorem~\ref{conj:main} and Theorem~\ref{conj:charge} in Section \ref{sec:equiv} and check that the highest degree terms agree with what we would expect.   In Section \ref{sec:s=2}, we give a combinatorial proof of Theorem \ref{conj:charge} at $s=2$, and in Section \ref{sec:Rnk}, we prove it for the $s_{(n)}$ coefficient in the Delta conjecture case. In Section~\ref{sec:conjs}, we give conjectural formulas for the Delta Conjecture symmetric function for $t$ degree at most $2$ in terms of skewing sums of Hall-Littlewood polynomials. Finally, in Section~\ref{sec:next}, we outline potential future research directions.

\subsection{Acknowledgments}

 We thank Brendon Rhoades for inspiring conversations at the start of this work, and Jim Haglund for helpful feedback after a talk on this material. We also thank William Graham and Amber Russell for helpful conversations on partial resolutions.

\section{Background}\label{sec:background}

We now recall some background and definitions on tableaux operations, cocharge and charge, and geometry related to the $\Delta$-Springer varieties.   We refer to \cite{Fulton} for the definition of the basic operation of \textbf{jeu de taquin rectification} on skew semistandard Young tableaux.  

\subsection{Tableaux and insertion}

We write partitions $\lambda=(\lambda_1,\ldots,\lambda_r)$ with their parts nonincreasing: $\lambda_1\ge \lambda_2\ge \cdots \ge \lambda_r$ and write $r=\ell(\lambda)$ for the length of $\lambda$.  We draw them in French notation, with $\lambda_i$ boxes in the $i$-th row from the bottom, and use the shorthand $(a^b)=(a,a,a,\ldots,a)$ to denote the $b\times a$ rectangular partition with $b$ parts of size $a$.  A \textbf{semistandard Young tableau (SSYT)} of \textbf{shape} $\lambda$ is a filling of the boxes of $\lambda$ that weakly increases across rows and strictly increases up columns. As stated in the introduction, we write $\SSYT(\mu)$ for the set of semistandard Young tableaux of \textbf{content} $\mu$ (and any shape).

The \textbf{reading word} of a tableau is the word formed by concatenating the rows from top to bottom.  For instance, the reading word of the battery-powered tableau in Figure \ref{fig:battery} is $$433111222442311.$$ 

The \textbf{RSK insertion} or \textbf{row bumping} of a letter $i$ into a tableau $T$ is the tableau $T'$ formed by inserting $i$ into the bottom row $R_1$ of $T$, where it is placed at the end if $i$ is greater than or equal to every element of $R_1$ and otherwise it replaces the leftmost entry $m$ of $R_1$ that is greater than $i$.  Then $m$ is inserted into the second row $R_2$ in the same manner, and so on until the process is complete and a new entry is added.  RSK insertion is reversible given the final bumped entry \cite{Fulton}, and we call the reverse process \textbf{unbumping}.

We also say the \textbf{RSK insertion} of a tableau $B$ into a tableau $D$ (such as in the case of a battery $B$ and device $D$) is the tableau $T'$ formed by inserting the letters of the reading word of $B$ one at a time into $D$.  We write $T'=D\cdot B$.  It is well-known (see \cite{Fulton}) that $D\cdot B$ is equal to the jeu de taquin rectification of the skew tableau formed by placing $B$ down-and-right of $D$.  We use this equivalence implicitly in this paper.

Two words are \textbf{Knuth equivalent} if their RSK insertions (one letter at a time inserted into the empty tableau from left to right) are equal.  

A \textbf{horizontal strip} is a skew shape in which no two boxes appear in the same column.  It is known that RSK inserting a nondecreasing sequence into a tableau $T$ extends the shape of $T$ by a horizontal strip.

\subsection{Symmetric functions}

We work in the ring of symmetric functions over $\mathbb{Q}$ in the countably infinite set of variables $x_1,x_2,x_3,\ldots$, which we often simply abbreviate as $x$.  We refer to \cite{Sagan} for the definitions of the \textbf{Schur functions} $s_\lambda(x)$ and the \textbf{elementary symmetric functions} $e_\lambda(x)$.


We recall that the \textbf{Hall inner product} is the symmetric inner product $\langle,\rangle$ on the space of symmetric functions for which $\langle s_\la,s_\mu\rangle =\delta_{\lambda\mu}$.  We write $f^\perp$ for the adjoint operator to multiplication by $f$ with respect to the Hall inner product; that is, $$\langle f^\perp( g),h\rangle =\langle g,f\cdot h \rangle. $$ 

It is known that $s_\mu^\perp s_\nu=s_{\nu/\mu}$.  We now observe a representation theoretic meaning of the operator $s_\mu^\perp$ (our statement can essentially be found in different language in \cite{Sagan}, and we include details and proof here for completeness).  In the below statement, the $V_\mu$-\textit{isotypic component} of an $S_n$-module $W$ is the sum of all copies of the irreducible Specht module $V_\mu$ in the decomposition of $W$ into irreducibles.

\begin{lemma}\label{lem:Skewing}
    Given $W$ an $S_n$-module, $S_{n-m}\times S_{m}$ a Young subgroup, and a partition $\mu\vdash m$, then
    \[
    s_\mu^\perp\Frob(W) = \frac{1}{\dim(V_\mu)}\Frob(W^{V_\mu})
    \]
    where $W^{V_\mu}$ is the $V_\mu$-isotypic component of the restriction of $W$ to an $S_{m}$-module, whose Frobenius character is taken as an $S_{n-m}$-module.
\end{lemma}

\begin{proof}
    By linearity, it suffices to check the lemma for $W = V_\nu$ where $\nu\vdash n$. In this case,
    \[
    \mathrm{Res}^{S_n}_{S_{n-m}\times S_{m}}(V_\nu) = \bigoplus_{\lambda\vdash m} V_{\nu/\lambda}\otimes V_{\lambda}
    \]
    where $V_{\nu/\lambda}$ is the skew Specht module corresponding to $\nu/\lambda$. Then the $V_\mu$-isotypic component of $V_\nu$ is $(V_{\nu/\mu})^{\oplus \dim(V_\mu)}$.   
    The formula follows since $s_\mu^\perp \Frob(V_\nu) = s_\mu^\perp s_\nu = s_{\nu/\mu}$.
\end{proof}


Also recall the \textbf{omega involution} on symmetric functions which may be defined as the unique linear operator $\omega$ such that $\omega(s_\lambda) = s_{\lambda^*}$, where $\lambda^*$ is the \textbf{conjugate partition} of $\lambda$.

Given a symmetric function $f(x;q)$ with coefficients in $\mathbb{Q}[q]$, we have the $q$-reversal operator $\revq$ which reverses the coefficients of $f$ as a polynomial in $q$. Precisely, if $f(x;q)$ has $q$ degree $d$ as a polynomial in $q$ with symmetric function coefficients, then $\revq(f(x;q)) = q^d f(x;1/q)$.

\subsection{Charge and cocharge}

We first define cocharge on words, using the reading word of the tableau $T$ in Figure \ref{fig:battery} as a running example: $$4_{\phantom{1}}3_{\phantom{1}}3_{\phantom{1}}1_{\phantom{1}}1_{\phantom{1}}1_{\phantom{1}}2_{\phantom{1}}2_{\phantom{1}}2_{\phantom{1}}4_{\phantom{1}}4_{\phantom{1}}2_{\phantom{1}}3_{\phantom{1}}1_{\phantom{1}}1.$$ 
The \textbf{first cocharge subword} is formed by searching right to left in the reading word for a $1$, then continuing from that position to search for a $2$ (wrapping around the end cyclically if necessary), and so on until we have reached the largest letter of the word:  $$\mathbf{4}_{\phantom{1}}{\color{gray}3}_{\phantom{1}}\mathbf{3}_{\phantom{1}}{\color{gray}1_{\phantom{1}}1_{\phantom{1}}1_{\phantom{1}}2_{\phantom{1}}2_{\phantom{1}}2_{\phantom{1}}4_{\phantom{1}}4_{\phantom{1}}}\mathbf{2_{\phantom{1}}}{\color{gray}3_{\phantom{1}}1_{\phantom{1}}}\mathbf{1.}$$
The \textbf{cocharge labeling} of a permutation is computed by searching right to left cyclically as before, labeling the entries $1,2,3,\ldots$ in order, and starting by labeling the $1$ with a $0$ and incrementing the label if and only if the next entry is to the left of the previous:
\[\mathbf{4}_3
{\color{gray}
3_{\phantom{1}}
}
\mathbf{3}_2
{\color{gray}
1_{\phantom{1}
}
1_{\phantom{1}}
1_{\phantom{1}}
2_{\phantom{1}}
2_{\phantom{1}}
2_{\phantom{1}}
4_{\phantom{1}}
4_{\phantom{1}}}
\mathbf{2}_1
{\color{gray}
3_{\phantom{1}}
1_{\phantom{1}}
}
\mathbf{1}_0.
\]
We then similarly find and label the \textit{second cocharge subword} among the unlabeled letters:
\[
{\color{gray}
4_3
}
\mathbf{3}_2
{\color{gray}
3_2
1_{\phantom{1}}
1_{\phantom{1}}
1_{\phantom{1}}
2_{\phantom{1}}
2_{\phantom{1}}
}
\mathbf{2}_{1}
{\color{gray}
4_{\phantom{1}}
}
\mathbf{4}_{2}
{\color{gray}
2_1
3_{\phantom{1}}
}
{\bf 1}_{0}
{\color{gray}
1_0.
}
\]
We continue to iterate this process on the unlabeled letters until all have been labeled:
\[
4_3
3_2
3_2
1_0
1_0
1_0
2_0
2_{0}
2_{1}
4_{1}
4_{2}
2_1
3_0
1_0
1_0.
\]
In Figure \ref{fig:battery}, the cocharge labels on the reading word elements are shown in the corresponding squares at right.   The \textbf{charge labels} are placed in the same order as cocharge labels except we increment when the next element is to the \textbf{right} of the previous.

  The \textbf{cocharge} (resp.\ \textbf{charge}) of $T$, written $\cc(T)$ and $\ch(T)$ respectively, is the sum of the cocharge (resp.\ charge) labels of its reading word.
Therefore, the cocharge of the word above is $3+2+2+1+1+2+1=12$.

Cocharge and charge are invariant under bumping: we have $\ch(D\cdot B)=\cc(T')$ and $\ch(D\cdot B)=\ch(T')$ where $T'$ is the insertion of $B$ into $D$.  This is because RSK insertion preserves the \textit{Knuth equivalence class} of the reading word \cite{Fulton}, and cocharge and charge are invariant under Knuth equivalence \cite{LascouxSchutzenberger}.

The maximum possible cocharge of a semistandard Young tableau of a given content $\nu$ occurs in the unique such tableau that has shape $\nu$ as well.  In this case, the cocharge label of each of the $\nu_i$ entries in the $i$-th row is $i-1$.  This leads to the following definition, which we use frequently throughout.

\begin{definition}
    We define the partition statistic $$\mathbf{n}(\la)=\sum_i (i-1)\la_i.$$ 
\end{definition}

\subsection{Hall-Littlewood polynomials}

We recall the Hall-Littlewood polynomials, which are symmetric functions with coefficients in a parameter $q$.
Given a partition $\mu$ of $n$, the \textbf{transformed Hall-Littlewood polynomial} $H_\mu(x;q)$ is the symmetric function with Schur expansion given by the charge statistic,
\begin{equation}\label{eq:HL-Charge}
H_\mu(x;q) = \sum_{T\in \SSYT(\mu)} q^{\ch(T)} s_{\sh(T)}.
\end{equation}
Alternatively, applying the $\revq$ operator we get the \textbf{modified Hall-Littlewood polynomial}, with Schur expansion given by the cocharge statistic,
\begin{equation}
    \widetilde{H}_\mu(x;q)= \revq(H_\mu(x;q)) = \sum_{T\in \SSYT(\mu)} q^{\cc(T)} s_{\sh(T)}.
\end{equation}
As mentioned in the introduction, the modified Hall-Littlewood polynomial $\widetilde{H}_\mu(x;q)$ is the graded Frobenius character of $R_\mu$, the cohomology ring of the Springer fiber $\sB_\mu$, which we define in the next subsection.

\subsection{Springer fibers and \texorpdfstring{$\Delta$}{}-Springer varieties}

Let $G = GL_K(\bC)$, let $B$ be the Borel subgroup of invertible upper triangular matrices, and let $\sB(K) = G/B$ be the complete flag variety, which may be identified with the space of complete flags $\sB(K) = \{F_\bullet = (F_1\subset F_2\subset\cdots \subset F_K)\mid \dim(F_i) = i, F_K = \bC^K\}$. Let $\cN$ be the nilpotent cone of $K\times K$ nilpotent matrices. 

The group $G$ acts on $\cN$ via the adjoint action, $Ad(g)x \coloneqq gxg^{-1}$.
For $x\in \cN$ nilpotent, we write $\JT(x)$ for the Jordan type of $x$, which is the partition of $K$ recording the Jordan block sizes of $x$ in Jordan canonical form. The set of all $x\in \cN$ with a fixed Jordan type $\mu$ is an orbit of $\cN$ under the adjoint action of $G$, which we denote by $\cO_\mu$.

Given $x\in \cN$, the \textbf{Springer fiber} associated to $x$ is
\[
\sB_x = \{F_\bullet\in \sB(K)\mid xF_i \subseteq F_i\text{ for all }i\}.
\]
  The isomorphism type of $\sB_x$ only depends on $\JT(x)$, and thus we may write $\sB_\mu$ for any $x\in \cO_\mu$.
  
  Springer discovered that these varieties have the remarkable property that the symmetric group $S_K$ acts on the cohomology ring $H^*(\sB_\mu;\bQ)$ and (in Lie type A) the top nonzero cohomology group is an irreducible Specht module,
\[
H^{top}(\sB_\mu;\bQ) \cong V_\mu.
\]
More generally, Hotta and Springer~\cite{HottaSpringer} proved that
\[
\grFrob(H^*(\sB_\mu;\bQ)) = \widetilde{H}_\mu(x;q).
\]

In~\cite{GLW}, Levinson, Woo, and the second author introduced the \emph{$\Delta$-Springer varieties} that generalize the Springer fibers and give a geometric realization of the symmetric function in the Delta Conjecture at $t=0$.

Let $n,\lambda,s$ be as in Definition~\ref{def:Lambda}, and let $K = |\Lambda| = k + (n-k)s = n+(n-k)(s-1)$. Let $x$ be a nilpotent $K\times K$ matrix with Jordan type $\Lambda$, and let $P$ be a parabolic subgroup of $G = GL_K$ with block sizes $\alpha = (1^n,(n-k)(s-1))$, so that $\sP = G/P$ corresponds to partial flags $(F_1\subset F_2\subset\cdots \subset F_n \subset F_{n+1})$ with $\dim(F_i) = i$ for $i\leq n$ and $F_{n+1} = \bC^K$. The \textbf{$\Delta$-Springer varieties} are defined to be 
\[
Y_{n,\la,s} \coloneqq \{F_\bullet \in \sP \mid xF_i\subseteq F_i\text{ for all }i\text{ and }  F_n \supseteq \mathrm{im}(x^{n-k})\}.
\]
Recall that we write $k=|\lambda|$.  When $k=n$ (and $s$ is arbitrary), $Y_{n,\la,s} \cong \sB_\lambda$, so these varieties generalize the Springer fibers. 

Levinson, Woo, and the second author proved that the $\Delta$-Springer varieties $Y_{n,\lambda,s}$ have several geometric and combinatorial properties that generalize those of Springer fibers:
\begin{itemize}
    \item $Y_{n,\lambda,s}$ is equidimensional of dimension $\mathbf{n}(\la) + (n-k)(s-1)$.
    \item There is an $S_n$ action on $H^*(Y_{n,\lambda,s};\bQ)$.
    \item The top cohomology group is a skew Specht module $H^{top}(Y_{n,\lambda,s};\bQ) \cong V_{\Lambda/((n-k)^{s-1})}$.
    \item $H^*(Y_{n,\lambda,s})$ has a presentation as a quotient of the polynomial ring $\bZ[x_1,\dots, x_n]$ which coincides with the ring $R_{n,\la,s}$ introduced in \cite{Griffin-Thesis}. In the special case $\la = (1^k)$ and $s=k$, the cohomology ring coincides with the generalized coinvariant rings of Haglund, Rhoades, and Shimozono, $H^*(Y_{n,(1^k),k};\bQ) = R_{n,k}$.
\end{itemize}
Notably, in the special case when $\la = (1^k)$ and $s=k$, then
\[
\grFrob(H^*(Y_{n,(1^k),k};\bQ)) = \grFrob(R_{n,k})=\omega\circ \revq(\Delta'_{e_{k-1}}e_n|_{t=0}),
\]
so $Y_{n,(1^k),k}$ gives a geometric realization of the symmetric function in the Delta Conjecture at $t=0$ (up to a minor twist).

\subsection{Rational smoothness and intersection cohomology}

\begin{definition}
    A complex variety $X$ of complex dimension $n$ is \textbf{rationally smooth} if either of the following equivalent conditions is satisfied:
    \begin{enumerate}
        \item For all $x\in X$, $H^i(X,X-x;\bQ)$ is $\bQ$ for $i=2n$ and $0$ for $i\neq 2n$.
        \item For all $x\in X$, the local intersection cohomology is trivial, meaning $IH^i_x(X;\bQ) = \bQ$ for $i=0$ and $0$ for $i\neq 0$.
    \end{enumerate}
\end{definition}
Here $IH^*_x$ is the \emph{middle} local intersection cohomology, see \cite{Goresky-MacPherson}. See \cite{Borho-MacPherson} for a proof of the fact that (1) and (2) above are equivalent. We do not define intersection cohomology here, but the essential property of local intersection cohomology that we need is that for $x\in \cO_\mu$,  
\begin{equation}
    \sum_k q^k \dim(IH_x^{2k}(\overline{\cO}_\nu;\bQ)) = q^{-\mathbf{n}(\nu)} \widetilde{K}_{\nu,\mu}(q),\label{eq:IH-Of-Orbit}
\end{equation}
which is a result due to Lusztig \cite{Lusztig-IH}. See also \cite{shoji} for more details and related results.
In particular, \eqref{eq:IH-Of-Orbit} reflects the fact that 
\begin{equation}\label{eq:orbit-closure}
\overline{\cO}_\nu = \bigcup_{\mu\preceq \nu} \cO_\mu,
\end{equation}
where $\preceq$ is dominance order on partitions of the same size, defined by $\mu\preceq \nu$ if $\mu_1 + \cdots +\mu_i \leq \nu_1 + \cdots + \nu_i$ for all $i$ \cite{McGovern}.

We will need the next fact, which follows easily from the Relative K\"unneth Formula for the local cohomology of a product space.
\begin{lemma}\label{lem:rat-smooth-fiber}
    Suppose $f:X\to Y$ is a fiber bundle with fiber $F$ such that both $F$ and $Y$ are rationally smooth. Then $X$ is also rationally smooth.
\end{lemma}

\subsection{Borho and MacPherson's partial resolutions}

Let $P$ be a parabolic subgroup, and let $\sP = G/P$ be the corresponding partial flag variety. Let $L$ be the Levi subgroup associated to $P$, let $\cN_L$ be the nilpotent cone of $L$, and finally let $\fp = \fl \oplus \fn$ be the Levi decomposition of $\fp = \Lie(P)$, where $\fl = \Lie(L)$ and $\fn$ is the nilradical of $\fp$.

Explicitly, $P$ is the set of invertible block upper triangular
matrices with block sizes given by some composition $\alpha$ of $K$, $G/P$ is the variety of partial flags $(V_1\subseteq V_2\subseteq \cdots \subseteq V_\ell)$ of $\bC^K$ with $\dim(V_i/V_{i-1}) = \alpha_i$ for all $i$, and $L$ is the subgroup of invertible block diagonal matrices with block sizes given by $\alpha$. The Lie algebra $\cN_L$ is the set of nilpotent block diagonal matrices, $\fp$ is the set of block upper triangular matrices, and $\fl$ is the set of block diagonal matrices, with block sizes given by the parts of $\alpha$. 

\begin{example}
    For $K=7$ and $P$ the parabolic subgroup with block sizes $\alpha = (3,1,1,2)$, the Levi decomposition $\fp = \fl \oplus \fn$ has the form
    \[
    \begin{bmatrix} 
    \ast & \ast & \ast & \ast & \ast & \ast & \ast\\
    \ast & \ast & \ast & \ast & \ast & \ast & \ast \\
    \ast & \ast & \ast & \ast & \ast & \ast & \ast \\
    0    &    0 &    0 & \ast & \ast & \ast & \ast \\
    0    & 0    & 0    &  0   & \ast & \ast & \ast \\
     0   & 0    & 0    & 0    & 0    & \ast & \ast \\
     0   & 0    & 0    & 0    & 0    & \ast & \ast 
    \end{bmatrix}
     = 
    \begin{bmatrix} 
    \ast & \ast & \ast &    0 &    0 &    0 &    0 \\
    \ast & \ast & \ast &    0 &    0 &    0 &    0 \\
    \ast & \ast & \ast &     0&    0 &    0 &    0 \\
    0    &    0 &    0 & \ast &    0 &    0 &    0 \\
    0    & 0    & 0    &  0   & \ast &    0 &    0 \\
     0   & 0    & 0    & 0    & 0    & \ast & \ast \\
     0   & 0    & 0    & 0    & 0    & \ast & \ast 
    \end{bmatrix}
     \oplus
     \begin{bmatrix} 
    0    &0     & 0    & \ast & \ast & \ast & \ast\\
    0    &0     & 0    & \ast & \ast & \ast & \ast \\
    0    &0     & 0    & \ast & \ast & \ast & \ast \\
    0    &    0 &    0 & 0    & \ast & \ast & \ast \\
    0    & 0    & 0    &  0   & 0    & \ast & \ast \\
     0   & 0    & 0    & 0    & 0    & 0    & 0    \\
     0   & 0    & 0    & 0    & 0    & 0    & 0    
    \end{bmatrix}
    \]
\end{example}

Borho and MacPherson defined the \emph{partial resolutions} of the nilpotent cone, defined by
\[
\xi:\widetilde{\mathcal{N}}^P \coloneqq G\times_P (\mathcal{N}_L + \mathfrak{n}) \to \mathcal{N},
\]
where $\xi(g,x) = \Ad(g)x=gxg^{-1}$. Here, the $\times_P$ notation denotes that we are taking the quotient of the product space by the $P$ action $p\cdot (g,x) = (gp^{-1},\Ad(p)x)$.
In type A, $\widetilde{\cN}^P$ has the following alternative description in terms of partial flags,
\[
\widetilde{\cN}^P \cong \{(F_\bullet,x)\in G/P\times \cN \mid xF_i\subseteq F_i \text{ for all }i\},
\]
where $\xi$ is the projection onto the second factor.  In particular, when $P=B$ then $\cN_L = 0$ and $\fn$ are the strictly-upper triangular matrices, and hence we recover the usual Springer resolution, which we denote by $\pi : \widetilde{\cN} = \widetilde{\cN}^B\to \mathcal{N}$.

Given $t\in \cN_L$, let $\cO_t = \Ad(L)t$. Let $y = (1,t + u)\in \widetilde{\cN}^P$ for arbitrary $u\in \fn$. 
 The subspaces
\[
\cO_y \coloneqq G\times_P (\cO_t + \fn)
\]
partition $\widetilde{\cN}^P$ as $t$ varies over $t\in \cN_L$. Since $\cO_y$ is a fiber bundle over $G/P$ with fiber $\cO_t+\fn$, taking the closure we have 
\begin{equation}\label{eq:Oy-closure}
    \overline{\cO}_y = G\times_P(\overline{\cO}_t + \fn),
\end{equation}
which can be seen by taking the closure on each trivializing open subset of $G/P$.

The usual Springer resolution $\pi:\widetilde{\cN}\to \cN$ factors through $\xi$. Letting $\xi_y$ be the restriction of $\xi$ to $\overline{\cO}_y$, we have the following commutative diagram.
\[
\begin{tikzcd}
     & \widetilde{\cN}\arrow[d,"\eta"]\arrow[bend left =40,dd,"\pi"]\\
    \overline{\cO}_y \arrow[r, hook] \arrow[d,"\xi_y"] & \widetilde{\cN}^P\arrow[d,"\xi"]\\
    \cN \arrow[r,"="] &\cN
\end{tikzcd}
\]
Given $x\in \cN$, the \textbf{generalized Springer fiber} is $\sP_x^y\coloneqq \xi_y^{-1}(x) = \overline{\cO}_y\cap \xi^{-1}(x)$. Note that the ordinary Springer fibers $\sB_x$ are recovered when $P=B$ is the full Borel subgroup (and $y=0$). 

In type A, the variety $\sP_x^y$ can alternatively be described in terms of partial flags as follows. Given $(g,x')\in \sP_x^y$, let $F_\bullet\in \sP$ be the partial flag corresponding to $gP$. That is, letting $e_1,\dots, e_K$ be the standard basis vectors of $\bC^K$ (not to be confused with the elementary symmetric polynomials), then $F_i = \mathrm{span}\{ge_1,\dots,ge_{\alpha_1+\cdots + \alpha_i}\}$. Since $(g,x')\in \sP_x^y$, then by definition $x = \Ad(g)x'$, and it can be checked that $xF_i\subseteq F_i$ for all $i$. Thus, $x$ induces a nilpotent endomorphism of $F_i/F_{i-1}$ for all $i$, which we denote by $x|_{F_i/F_{i-1}}$. Letting $t = t_1+\cdots +t_\ell$ be the block decomposition of $t$, it then follows from \eqref{eq:Oy-closure} that
\begin{equation}\label{eq:P-as-flags}
\sP_x^y \cong \{F_\bullet\in \sP \mid xF_i \subseteq F_i \text{ for all }i \text{ and } \JT(x|_{F_i/F_{i-1}}) \preceq \JT(t_i)\text{ for all }i\}.
\end{equation}

Let $\sB(L)_t$ be the Springer fiber of $t$ in the flag variety $\sB(L)\cong \sB(\alpha_1)\times \cdots \times \sB(\alpha_\ell)$ for the group $L$.  In other words,  
\[
\sB(L)_t \cong (\sB(\alpha_1))_{t_1}\times \cdots \times (\sB(\alpha_\ell))_{t_\ell}.
\]
Borho and MacPherson showed that $\eta^{-1}(y)\cong \sB(L)_t$. We write $d_y = \dim_\bC(\eta^{-1}(y)) = \dim_\bC(\sB(L)_t)$.

Let $\rho(t,1)$ be the irreducible representation of $W_L = S_{\alpha_1}\times \cdots \times S_{\alpha_\ell}$ on $H^{top}(\sB(L)_t;\bQ)$. In other words, $\rho(t,1) \cong V_{\JT(t_1)}\otimes \cdots \otimes V_{\JT(t_\ell)}$ as a $W_L$-module. Given a $W$-module $V$, recall that $V^{\rho(t,1)}$ is the isotypic component corresponding to $\rho(t,1)$ of the restriction of $V$ to a $W_L$-module. Observe that the ``partial Weyl group'' $W^P = N_G(L)/L$ of permutations of the blocks of $L$ of equal size acts on $V^{\rho(t,1)}$.

\begin{thm}[\cite{Borho-MacPherson}]\label{thm:BM-iso}
    The partial Weyl group $W^P = N_G(L)/L$ acts on $H^*(\sP_x^y;\bQ)$, and the Springer action of $W = S_K$ restricts to an action of $W^P$ on $H^*(\sB_x;\bQ)^{\rho(t,1)}$. 
    
    Furthermore, if $\overline{\cO}_y$ is rationally smooth at all points of $\sP_x^y$, then there is a graded isomorphism of $W^P$-modules
    \[
        H^i(\sP_x^y;\bQ) \otimes H^{2d_y}(\sB(L)_t;\bQ) \cong H^{i+2d_y}(\sB_x;\bQ)^{\rho(t,1)}
    \]
    for all $i$, where $W^P$ acts trivially on the second factor of the tensor product.
\end{thm}

\section{Proof of the main theorm}\label{sec:MainThm}

In this section, we prove Theorem~\ref{thm:skew-formula} using the geometry of Borho--MacPherson partial resolutions. Readers interested in the combinatorial applications of the formula may skip to Section~\ref{sec:equiv}.

We begin with a technical lemma that will help us apply Theorem \ref{thm:BM-iso} to our setting of $\Delta$-Springer varieties.

\begin{lemma}\label{lem:JT-rect}
    Let $\cO_\mu$ be the $\Ad(G)$-orbit of elements of $\cN$ with Jordan type $\mu$. For $\mu$ a rectangular partition $\mu = (a^b)$ (so $n=ab$) then 
    \[
    \overline{\cO}_\mu = \bigcup_{\substack{\nu\vdash n,\\ \nu_1\leq a}} \cO_\nu.
    \]
\end{lemma}

\begin{proof}
    By \eqref{eq:orbit-closure}, the statement of the lemma is equivalent to: $\nu\preceq (a^b)$ if and only if $\nu_1\leq a$. In the forward direction, if $\nu\preceq (a^b)$, then $\nu_1\leq a$ follows by definition of dominance order. For the converse, suppose that $\nu_1\leq a$, so that $\nu_i\leq a$ for all $i$, since $\nu$ is a partition. Then $\nu_1+\cdots + \nu_i \leq a\cdot i$, which is the sum of the first $i$ parts of $(a^b)$, so the lemma follows.
\end{proof}

\begin{lemma}\label{lem:JT-containment}
    Let $x$ be a nilpotent $K\times K$ matrix such that $\JT(x) = \Lambda_{n,\la,s}$, and let $F_\bullet\in Y_{n,\la,s}$. Letting $x|_{\bC^K/F_n}$ be the nilpotent endomorphism of $\bC^K/F_n$ induced by $x$, we have $\JT(x|_{\bC^K/F_n}) \subseteq ((n-k)^s)$. 
\end{lemma}

\begin{proof}
    The statement of the lemma is independent of conjugating $x$ by an invertible matrix. We choose $x$ to be of the following form: Label the Young diagram of $\Lambda_{n,\la,s}$ with the standard basis vectors $e_1,\dots,e_K$ in order from right to left along each row, bottom to top. 
    
    For example, when $n=5$, $\la=(2,1)$, and $s=3$, then we have the labeling
$$\begin{ytableau}
    e_9 & e_8 \\
    e_7 & e_6 & e_5 \\
    e_4 & e_3 & e_2 & e_1
\end{ytableau}.$$

    Define $x$ to be the $K\times K$ matrix such that $xe_i = e_j$ if $e_i$ is in the cell immediately to the left of $e_j$, and $xe_i = 0$ if $e_i$ is in the right-most cell in its row. Then $\mathrm{im}(x^{n-k})$ is the span of the $k$ vectors in the cells of $\Lambda$ that are in columns $>n-k$ from the left in the Young diagram (in this case, $e_1,e_2,e_5$). Since $F_\bullet\in Y_{n,\la,s}$, then $F_n\supseteq \mathrm{im}(x^{n-k})$. Thus, the Jordan type of $x|_{\bC^K/F_n}$ is contained in $\JT(x|_{\bC^K/\mathrm{im}(x^{n-k})})= ((n-k)^s)$.
\end{proof}

\begin{lemma}\label{lem:PisY}
    Let $\alpha = (1^n,K-n)$, $\JT(x) = \Lambda_{n,\la,s}$, and $\JT(t_i) = (1)$ for $i\leq n$ and $\JT(t_{n+1}) = ((n-k)^{s-1})$. Then $\sP_x^y \cong Y_{n,\lambda,s}$.
\end{lemma}

\begin{proof}
    Given $F_\bullet\in \sP$, then by~\eqref{eq:P-as-flags}, $F_\bullet\in \sP_x^y$ if and only if $\JT(x|_{F_i/F_{i-1}}) \preceq \JT(t_i)$ for all $i$. For $\alpha=(1^n,K-n)$, this is equivalent to $\JT(x|_{\bC^K/F_n}) \preceq ((n-k)^{s-1})$, which by Lemma~\ref{lem:JT-rect} is equivalent to $\JT(x|_{\bC^K/F_n})_1 \leq n-k$. 
    
    We claim that $\JT(x|_{\bC^K/F_n})_1\leq n-k$ if and only if $F_\bullet\in Y_{n,\la,s}$. Indeed, the reverse direction follows by Lemma~\ref{lem:JT-containment}. We prove the forward direction by proving the contrapositive: Suppose $F_\bullet\notin Y_{n,\la,s}$, meaning $\mathrm{im}(x^{n-k})\not\subseteq F_n$. Then there exists some nonzero $v\in \mathrm{im}(x^{n-k})\setminus F_n$. The transpose operator $x^t$ is the linear operator defined by $x^te_i = e_j$ if and only if $e_i$ is in the cell immediately to the right of $e_j$, and $x^t e_i = 0$ if $e_i$ is in the first column. Then $v,x^tv,(x^t)^2v,\dots, (x^t)^{n-k}v\notin F_n$ since $xF_n \subseteq F_n$. Furthermore, it can be checked that they are linearly independent vectors. Choosing a basis of $\bC^K/F_n$ that includes these $n-k+1$ vectors shows that $\JT(x|_{\bC^K/F_n})_1 \geq n-k+1$. Thus, the claim is proved, and it follows that $F_\bullet\in \sP_x^y$ if and only if $F_\bullet \in Y_{n,\la,s}$.
\end{proof}

\begin{lemma}\label{lem:rat-smooth}
    Let $\alpha = (1^n,K-n)$, $\JT(x) = \Lambda_{n,\la,s}$, and $\JT(t_i) = (1)$ for $i\leq n$ and $\JT(t_{n+1}) = ((n-k)^{s-1})$. In this case, the hypotheses of Theorem~\ref{thm:BM-iso} hold: $\overline{\cO}_y$ is rationally smooth at all points of $\sP_x^y$.
\end{lemma}

\begin{proof}
    Given $(g,x')\in \sP_x^y$, then $x' = \Ad(g^{-1})x \in \overline{\cO}_t + \fn$. Let $F_\bullet$ be the partial flag corresponding to $gP$, meaning $F_i = \mathrm{span}\{ge_1,\dots, ge_i\}$ for $i\leq n$.   By Lemma \ref{lem:JT-containment}, we have $\JT(x|_{\bC^K/F_n})\subseteq ((n-k)^s)$.
    Since $x' = \Ad(g^{-1}) x$, we have a commutative diagram
    \[
    \begin{tikzcd}
        \bC^K/F_n \arrow[r,"x"] & \bC^K/F_n\\
        \bC^K/\mathrm{span}\{e_1,\dots, e_n\}\arrow[r,"x'"]\arrow[u,"g"] & \bC^K/\mathrm{span}\{e_1,\dots, e_n\}\arrow[u,"g"]
    \end{tikzcd}
    \]
    which implies that $\JT(x|_{\bC^K/F_n}) = \JT(x'|_{\bC^K/\mathrm{span}\{e_1,\dots,e_n\}})$.     
    Thus, the Jordan type of the last diagonal block of $x'$ has length at most $s$. Thus, 
    $x' \in \overline{\cO}_t\setminus Z + \fn$ where 
    \[
    Z \coloneqq \bigcup_{\substack{t' \in \overline{\cO}_t,\\ \ell(\JT(t'_{n+1})) > s}}\cO_{t'}.
    \]
    Note that $Z$ is a closed subvariety of $\overline{\cO}_t$. 

    We thus have $\sP_x^y \subseteq G\times_P(\overline{\cO}_t\setminus Z + \fn)$. We claim that, since $G\times_P(\overline{\cO}_t\setminus Z + \fn)$ is an open subset of $\overline{\cO}_y$, it suffices to show that $\overline{\cO}_t\setminus Z$ is rationally smooth. 
    Indeed, the space $G\times_P(\overline{\cO}_t\setminus Z + \fn)$ is homeomorphic to the fiber product $(G\times_P \overline{\cO}_t\setminus Z)\times_{\sP} (G\times_P\fn)$ (of fiber bundles over $\sP$). Since $G\times_P\fn$ is smooth, then it suffices to check that $\overline{\cO}_t\setminus Z$ is rationally smooth by Lemma~\ref{lem:rat-smooth-fiber}.

    Equivalently, we must show that $\overline{\cO}_{((n-k)^{s-1})}\setminus Z'$ is rationally smooth, where  
    \[
    Z' \coloneqq \bigcup_{\substack{\nu\vdash (n-k)(s-1),\\ \ell(\nu) > s}} \cO_\nu.
    \]
    Since local intersection cohomology only depends on a neighborhood of $u$ and $Z'$ is a closed subvariety, it suffices to show that for all $u\in \overline{\cO}_{((n-k)^{s-1})}\setminus Z'$,
    \[
    IH_u^i(\overline{\cO}_{((n-k)^{s-1})};\bQ) = 
    \begin{cases}
        \bQ & \text{ if }i=0\\
        0 & \text{ if }i\neq 0.
    \end{cases}
    \]

    Now, by Lemma \ref{lem:JT-rect}, $u\in \overline{\cO}_{((n-k)^{s-1})}\setminus Z'$ if and only if $u\in \cO_\mu$ for some $\mu$ such that $\mu_1\leq n-k$ and $\ell(\mu) \leq s$, which is equivalent to $\mu\subseteq ((n-k)^s)$. By \eqref{eq:IH-Of-Orbit}, for $u\in \cO_\mu$ we have
    \begin{equation}\label{eq:IH-comp}
    \sum_k q^k \dim(IH_u^{2k}(\overline{\cO}_{((n-k)^{s-1})};\bQ)) = q^{-\mathbf{n}((n-k)^{s-1})} \widetilde{K}_{((n-k)^{s-1}),\mu}(q).
    \end{equation}
    But for $\mu\vdash (n-k)(s-1)$ such that $\mu\subseteq ((n-k)^s)$, $\widetilde{K}_{((n-k)^{s-1}),\mu}(q) = q^{\mathbf{n}((n-k)^{s-1})}$ by Lemma~\ref{lem:K-q-power} below. Thus, the right-hand side of \eqref{eq:IH-comp} is $1$. Thus, $\overline{\cO}_{((n-k)^{s-1})}\setminus Z$ is rationally smooth, and $\overline{\cO}_y$ is rationally smooth at all points of $\sP_x^y$.
\end{proof}

\begin{lemma}\label{lem:K-q-power}
    Suppose $\mu\vdash ab$ for two positive integers $a$ and $b$ such that $\mu\subseteq (a)^{b+1}$. Then $\widetilde{K}_{(a^b),\mu}(q) = q^{\mathbf{n}((a)^{b})}$.
\end{lemma}

\begin{proof}
       There is a unique semistandard Young tableau $T$ with content $\mu$ and shape $(a)^{b}$. Indeed, if $T$ is such a semistandard Young tableau, since $\ell(\mu) \leq b+1$ and the shape of $T$ has $b+1$ rows, there is exactly one letter from $1,\dots, b+1$ missing from each column of $T$. Since $T$ has content $\mu$, then it has $a-\mu_i$ many columns that do not have the letter $i$, and there is only one way of arranging these columns into a semistandard Young tableau $T$ (the missing letters from each column must weakly decrease from left to right).  

       We now compute the cocharge of $T$.  We claim that the cocharge subscript of each letter is equal to one less than its row index.  Each letter $i$ is either in row $i$ or in row $i-1$ by construction; let the entries that are in their own row be called \textit{left entries} of $T$ and let the others be \textit{right entries}; notice that the left entries are separated from the right by a down-and-right path.  It follows that each cocharge subword consists of left entries $1,2,\ldots,i-1$ in their respective rows for some $i$, followed by a sequence of right entries $i,i+1,\ldots,b+1$ in rows $i-1,\ldots,b$ respectively.   Because the cocharge subword only wraps around at the jump from left to right entries, each subscript is equal to the row that the entry is in at every step. 
       Finally, it follows that $\cc(T)=a\cdot \binom{b}{2}=\mathbf{n}((a^b))$, and the result follows.  
\end{proof}

\begin{example}
    For $a=5$, $b=3$, and $\mu=(4,4,4,3)$, the tableau $T$ in the proof of Lemma \ref{lem:K-q-power} is $$\young({{\bf 3}}{{\bf 3}}444,{{\bf 2}}{{\bf 2}}{{\bf 2}}33,{{\bf 1}}{{\bf 1}}{{\bf 1}}{{\bf 1}}2).$$  The left entries are shown in boldface, and the right entries are normal font.  The cocharge subscript of each letter is one less than its row, and the cocharge is $5\cdot \binom{3}{2}=15$. 
\end{example}

We now can prove Theorem \ref{thm:skew-formula}, which we restate here.

\begin{skewthm}
We have
\begin{equation}\label{eq:s-perp}
\widetilde{H}_{n,\lambda,s}(x;q) = \frac{1}{q^{\binom{s-1}{2}(n-k)}} s_{((n-k)^{s-1})}^\perp \widetilde{H}_\Lambda(x;q)
\end{equation}
\end{skewthm}

\begin{proof}
Observe that for $P$ the parabolic of type $(1^n,K-n)$, then $W^P\cong S_n$. Combining Theorem~\ref{thm:BM-iso}, Lemma~\ref{lem:PisY}, and Lemma~\ref{lem:rat-smooth}, we have an isomorphism of graded $S_n$-modules (where $S_n$ acts trivially on the second tensor factor)
\[
H^i(Y_{n,\lambda,s};\bQ) \otimes H^{2d_y}(\sB(L)_t;\bQ) \cong H^{i+2d_y}(\sB_x)^{\rho(y,1)}.
\]
Recall that $d_y = \dim_{\mathbb{C}}(\eta^{-1}(y)) = \dim_{\mathbb{C}}(\sB(L)_t)$. Note that in this case, $\sB(L)_t \cong \sB_{t_{n+1}}$.
Since $\JT(t_{n+1}) = ((n-k)^{s-1})$, then $\dim(H^{2d_y}(\sB(L)_t;\bQ)) = \dim(V_{((n-k)^{s-1})})$. 

We have $\rho(y,1) \cong V_{(1)}\otimes \cdots \otimes V_{(1)}\otimes V_{(s-1)^{n-k}}$ as $W_L = S_1\times \cdots \times S_1\times S_{(s-1)(n-k)}$-modules. Thus, since $d_y=\mathbf{n}((n-k)^{s-1})=\binom{s-1}{2}(n-k)$, we have
\begin{equation}\label{eq:final-eqn}
\dim(V_{((n-k)^{s-1})}) \grFrob(H^*(Y_{n,\la,s};\bQ)) = q^{\binom{s-1}{2}(n-k)}\grFrob(V_{\Lambda}^{V_{((n-k)^{s-1})}}).
\end{equation}
Theorem~\ref{thm:skew-formula} then follows by rearranging and applying Lemma~\ref{lem:Skewing}.
\end{proof}

\begin{remark}
    In the proof of Theorem~\ref{thm:skew-formula} above, we have implicitly used the fact that the $S_n$ action on $H^*(Y_{n,\la,s};\bQ)$ here is the same as the one in~\cite{GLW}. The action defined in \cite{GLW} was by permutations of the first Chern classes of the tautological quotient line bundles $F_i/F_{i-1}$ for $i\leq n$. The fact that this matches the action of $W^P$ on $H^*(Y_{n,\la,s};\bQ)$ follows from the fact that it is compatible with the $W=S_K$ action defined by Borho and MacPherson on the Springer fiber $H^*(\sB_x;\bQ)$, which in type A is well known to be the same as the action of $S_K$ by permutations of the first Chern classes of the tautological line bundles, see~\cite{Brundan-Ostrik}.
\end{remark}

As an immediate corollary of Theorem~\ref{thm:skew-formula}, we see how the formula in \cite{GLW} for the top cohomology of $Y_{n,\la,s}$ follows immediately from the skewing formula.
\begin{corollary}[{\cite[Theorem 1.3]{GLW}}]
    We have an isomorphism of $S_n$ modules,
    \[
    H^{top}(Y_{n,\la,s};\bQ) \cong S^{\Lambda/((n-k)^{s-1})},
    \]
    where $S^{\Lambda/((n-k)^{s-1})}$ is the skew Specht module corresponding to the skew partition $\Lambda/((n-k)^{s-1})$.
\end{corollary}

\begin{proof}
    By Theorem~\ref{thm:skew-formula}, the top degree of $\widetilde{H}_{n,\la,s}(x;q) = \grFrob(H^*(Y_{n,\la,s};\bQ))$ is $$s_{((n-k)^{s-1})}^\perp s_{\Lambda} = s_{\Lambda/((n-k)^{s-1})},$$ which is the graded Frobenius character of $S^{\Lambda/((n-k)^{s-1})}$. 
\end{proof}

\section{Proofs of the cocharge and charge formulas}\label{sec:equiv}

In this section, we use Theorem \ref{thm:skew-formula} to prove Theorems~\ref{conj:main} and \ref{conj:charge}.

\subsection{Proof of Theorem \ref{conj:main}}

We now deduce the cocharge formula (Theorem \ref{conj:main}) from Theorem \ref{thm:skew-formula}.
In particular, we wish to show that 
\begin{equation}\label{eq:goal}
    s_{((n-k)^{s-1})}^\perp \widetilde{H}_\Lambda(x;q)=\sum_{T\in \mathcal{T}^+(n,\la,s)} q^{\cc(T)}s_{\sh^+(T)}.
\end{equation}  
Recall also the following skewing formula for applying an adjoint Schur operator to another Schur function:
\begin{equation}
    s_\lambda^\perp s_\mu=s_{\mu/\lambda}
\end{equation}

We will prove the following more general lemma, from which Equation \eqref{eq:goal} immediately follows.  Define a \textbf{generalized battery-powered tableau} with (not necessarily rectangular) battery shape $\rho$ and content $\mu$ to be a pair $(D,B)$ of semistandard Young tableaux such that $\sh(B)=\rho$ and the total content of $D\cup B$ is $\mu$.  Write $\mathcal{T}^+(\rho,\mu)$ to be the set of all such pairs $T=(D,B)$, and write $\sh^+(T)=\sh(D)$.

\begin{lemma}
      We have $$s_\rho^\perp \widetilde{H}_\mu(x;q)=\sum_{T\in \mathcal{T}^+(\rho,\mu)}q^{\cc(T)}s_{\sh^+(T)}$$ where $\cc(T)=\cc(D\cdot B)=\cc(D\cup B)$ where $D\cup B$ is formed by placing $B$ down-and-right of $D$. 
\end{lemma}

\begin{proof}
Let $\SSYT(\mu)$ be the set of semistandard Young tableaux of content $\mu$ (of any shape), and let $\SSYT(\nu,\mu)$ be the set of semistandard Young tableaux of shape $\nu$ and content $\mu$.  From the Lascoux-Sch\"utzenberger formula \eqref{eq:lascoux-schutzenberger} for Hall-Littlewood polynomials, the left hand side above expands as 
\begin{align}
    s_{\rho}^\perp \sum_{T\in \SSYT(\mu)}q^{\cc(T)}s_{\sh(T)} &= \sum_{\nu}\sum_{T\in \SSYT(\nu,\mu)}q^{\cc(T)}s_{\nu/\rho}\\
    &= \sum_{\nu}\sum_{T\in \SSYT(\nu,\mu)}\sum_{\eta} q^{\cc(T)}c^{\nu}_{\eta,\rho} s_{\eta} \label{eq:LR}
\end{align}
where $c^{\nu}_{\eta,\rho}$ is the Littlewood-Richardson coefficient. For any fixed SSYT $T$ of shape $\nu$, we may interpret $c^{\nu}_{\eta,\rho}$ as the number of pairs $(D,B)$ of semistandard Young tableaux of shapes $\eta$ and $\rho$ respectively such that $D\cdot B=T$ (see \cite{Fulton}).  Since $\cc$ is invariant under jeu de taquin and RSK insertion, we have $\cc(D\cup B)=\cc(D\cdot B)=\cc(T)$.  Thus the sum above becomes 
\begin{align*}
    \sum_{\nu}\sum_{T\in \SSYT(\nu,\mu)}\sum_{\substack{D\cdot B=T \\ \sh(B)=\rho}}q^{\cc(D\cup B)} s_{\sh(D)}&=\sum_{(D,B)\in \mathcal{T}^+(\rho,\mu)}q^{\cc(D\cup B)}s_{\sh(D)}\\
    &=\sum_{T\in \mathcal{T}^+(\rho,\mu)}q^{\cc(T)}s_{\sh^+(T)}
\end{align*}
as desired.   
\end{proof}

From line \eqref{eq:LR} above, setting $\mu=\Lambda_{n,(1^k),k}$ and $\rho=((n-k)^{k-1})$, we can also deduce \[
    \langle s_\mu, \,\omega\circ \revq(\Delta'_{e_{k-1}}e_n)|_{t=0}) = \frac{1}{q^{\binom{k-1}{2}(n-k)}} \sum_{\nu\vdash k(n-k+1)} c_{\mu,((n-k)^{k-1})}^\nu \widetilde{K}_{\nu,((n-k+1)^k)}(q).
    \]
    Corollary \ref{cor:LR-Kostka} follows immediately by applying the $\revq$ operator.

\subsection{Proof of Theorem \ref{conj:charge}}

We now deduce the charge version of the main result, Theorem \ref{conj:charge}, from Theorem \ref{conj:main}.  For any partition $\nu$, recall that $\mathbf{n}(\nu)=\sum_i (i-1)\nu_i$.

\begin{prop}\label{prop:q-degree}
    The maximum value of $\cc(T)$ for $T\in \mathcal{T}^+(n,\la,s)$ is $$\mathbf{n}(\lambda)+\binom{s}{2}(n-k).$$ Moreover, there is precisely one battery-powered tableau $T$ with this value of $\cc$ for each device shape $\nu$ with $\ell(\nu)\le s$ and where $\nu/\lambda$ is a horizontal strip (and no tableaux with this value of $\cc$ for other device shapes). 
\end{prop}

\begin{proof}
The maximal cocharge among all words of a given content $\Lambda$ occurs when each cocharge subword has its letters appearing in order from right to left, and in that case the cocharge is $\mathbf{n}(\Lambda)$.  For this to occur, the battery columns must be filled with $1,2,\ldots,s-1$ from bottom to top, for otherwise some entry of the battery $B$ would be to the right of the previous element in its cocharge subword.  The subwords starting at the $1$'s in the bottom of $B$ will then contain $1,2,\ldots,s$ from right to left, with the $s$ being in the device.  

For the cocharge subwords starting at $1$'s in the device $D$ to be in right to left order, $D$ must contain the unique tableau $D'$ of content $\lambda$ and shape $\lambda$ (with $\lambda_i$ entries $i$ in the $i$-th row for all $i$).  So, $D$ is formed by adding a horizontal strip of length $n-k$ labeled by $s$ to $D'$ such that the result is semistandard.  Thus there is one tableau of maximal cocharge for each shape of height $\le s$ formed by adding a horizontal strip to $\lambda$.

For such pairs $(D,B)$, we have $\cc(D,B)=\mathbf{n}(\Lambda)=\mathbf{n}(\lambda)+\binom{s}{2}(n-k)$, as desired.
\end{proof}

Dividing out by the factor $q^{\binom{s-1}{2}(n-k)}$, we obtain the following corollary.

\begin{corollary}\label{cor:top-degree}
The top $q$-degree of the polynomial on the right hand side of Theorem \ref{conj:main} is $d:=\mathbf{n}(\lambda)+(s-1)(n-k)$, and the coefficient of $q^d$ is $\sum s_\nu$ where the sum ranges over all partitions $\nu$ of $n$ with $\ell(\nu)\le s$ and $\nu/\lambda$ a horizontal strip.
\end{corollary}

The value $d$ matches with the formula given for the top degree of $\grFrob_q(R_{n,\la,s})$ in \cite{Griffin-Thesis}.  In \cite{GLW}, it was shown that the coefficient of $q^d$ is the \textit{skew Schur function} $s_{\Lambda/((n-k)^{s-1})}$.  A straightforward application of the Littlewood-Richardson rule shows that this agrees with our formula in Corollary \ref{cor:top-degree}, and we refer to \cite{GG} for details.

Finally, we show that Theorems \ref{conj:main} and \ref{conj:charge} are equivalent.  Taking the $q$-reversal of both sides of Theorem \ref{conj:main}, we have $$\revq \left(\widetilde{H}_{n,\lambda,s}\right) = \sum_{T\in \mathcal{T}^+(n,\lambda,s)} q^{\mathbf{n}(\lambda)+(n-k)(s-1)-\cc(T)+\binom{s-1}{2}(n-k)}s_{\sh^+(T)}.$$ Then the exponent $\mathbf{n}(\lambda)+(n-k)(s-1)-\cc(T)+\binom{s-1}{2}(n-k)$ is equal to $\mathbf{n}(\Lambda)-\cc(T)$, which is simply $\ch(T)$ by the definition of charge.  This gives Theorem \ref{conj:charge}.

\section{The case \texorpdfstring{$s=2$}{}}\label{sec:s=2}

In this section, we give a second proof of Theorem~\ref{conj:main} in the case when $s=2$ using combinatorial bijections and previously known formulas for $\widetilde{H}_{n,\la,s}$. We start by recalling the Hall-Littlewood expansion of $\widetilde{H}_{n,\la,s}$.

\subsection{Hall-Littlewood expansion}

In \cite{Griffin-HLExpansion}, it is shown that $\widetilde{H}_{n,\la,s}$ has the following expansion in terms of Hall-Littlewood polynomials.

\begin{equation}\label{eq:HL-expansion}
    \widetilde{H}_{n,\lambda,s}(X;q)=\revq\left(\sum_{\substack{\mu\vdash n,\\ \mu \supset \lambda,\\ \ell(\mu)\le s}} q^{\mathbf{n}(\mu/\lambda)}\sum_{\substack{\alpha=(\alpha_1,\ldots,\alpha_s)\vDash n,\\ \alpha\supset \lambda, \sort(\alpha)=\mu}} q^{\coinv(\alpha)} H_\mu(x;q)\right),
\end{equation} where $\alpha = (\alpha_1,\dots,\alpha_s)\vDash n$ indicates that $\alpha$ is a weak composition of $n$ with $s$ parts, $\mathbf{n}(\mu/\lambda)=\sum_i \binom{\mu_i'-\lambda_i'}{2},$ and $\mathrm{coinv}(\alpha)$ is the number of pairs $i<j$ with $\alpha_i<\alpha_j$.

Note that if $\alpha$ is a composition such that $\alpha\supset \la$,  then since $\lambda$ is a partition we also have $\sort(\alpha)\supset \lambda$.  Thus we can rearrange the summation above as 

\begin{equation}\label{eq:HL-expansion-comp}
 \widetilde{H}_{n,\lambda,s}(X;q)=\revq\left(\sum_{\substack{\alpha=(\alpha_1,\ldots,\alpha_s)\models n,\\ \alpha\supset \lambda}} q^{\mathbf{n}(\alpha/\lambda)+\coinv(\alpha)} H_{\sort(\alpha)}(x;q)\right)
\end{equation}
where the quantity $\mathbf{n}(\alpha/\lambda)$ above is defined to be $\mathbf{n}(\mu/\lambda)$ where $\mu=\sort(\alpha)$.

  Substituting \eqref{eq:HL-Charge} into \eqref{eq:HL-expansion-comp} yields 
\begin{equation}\label{eq:alpha-expansion}
    \revq\left(\widetilde{H}_{n,\lambda,s}(X;q)\right)=\sum_{\substack{\alpha=(\alpha_1,\ldots,\alpha_s)\models n,\\ \alpha\supset \lambda}}\sum_{T\in \SSYT(\sort(\alpha))} q^{\mathbf{n}(\alpha/\lambda)+\coinv(\alpha)+\ch(T)} s_{\sh(T)}.
\end{equation}

Thus, to prove Theorem \ref{conj:charge} it suffices to show that
\begin{equation}
    \sum_{T\in \mathcal{T}^+(n,\la,s)} q^{\ch(T)}s_{\sh^+(T)}=\sum_{\substack{\alpha=(\alpha_1,\ldots,\alpha_s)\models n,\\ \alpha\supset \lambda}}\sum_{U\in \SSYT(\sort(\alpha))} q^{\mathbf{n}(\alpha/\lambda)+\coinv(\alpha)+\ch(U)} s_{\sh(U)}.
\end{equation}
In particular, it suffices to find a shape-preserving bijection from $\mathcal{T}^+(n,\la,s)$ to $$\mathcal{A}(n,\lambda,s):=\{(\alpha,U)| \alpha=(\alpha_1,\ldots,\alpha_s)\models n, \alpha \supset \lambda, U\in \SSYT(\sort(\alpha))\}$$ such that, if $T\in \mathcal{T}^+(n,\la,s)$ maps to $(\alpha,U)\in \mathcal{A}(n,\lambda,s)$, then $\ch(T)=\ch(U)+\mathbf{n}(\alpha/\lambda)+\coinv(\alpha)$. In the next subsection, we find such a bijection in the case $s=2$.

\subsection{Combinatorial proof for \texorpdfstring{$s=2$}{}}  For the remainder of this section, let $\lambda=(\lambda_1,\lambda_2)$ be a partition of size $k$ with $\lambda_1\ge \lambda_2\ge 0$, and let $\mathrm{Alpha}(n,\lambda,2)$ be the set of all (weak) compositions $\alpha = (\alpha_1,\alpha_2)$ of size $n$ such that $\alpha\supset \lambda$.

\begin{definition}
For $\alpha \in \mathrm{Alpha}(n,\lambda,2)$, define $\varphi(\alpha)$ to be the composition formed by taking $\mathbf{n}(\alpha/\lambda)+\coinv(\alpha)$ boxes from the bottom row of $\sort(\alpha)$ and moving them to the top row.
\end{definition}

As a running example, let $n=11$, $\la = (3,1)$, $s=2$, and $\alpha = (5,6)$. Then $\mathbf{n}(\alpha/\la) + \coinv(\alpha) = 2+1=3$. Since $\sort(\alpha) = (6,5)$, then $\varphi(\alpha) = (3,8)$.

\begin{prop}\label{prop:alpha-2-row}
    The map $\varphi$ on compositions is a bijection from $\mathrm{Alpha}(n,\lambda,2)$ to itself. 
\end{prop}

\begin{proof}
  We first show that if $\alpha\in \mathrm{Alpha}(n,\lambda,2)$ then $\varphi(\alpha)\in \mathrm{Alpha}(n,\lambda,2)$.   Indeed, we have $\coinv(\alpha)=0$ or $1$ according to whether $\alpha_1\ge \alpha_2$ or $\alpha_1<\alpha_2$, and $\mathbf{n}(\alpha/\lambda)$ is the number of columns of $\alpha$ to the right of column $\lambda_1$ containing two squares.  Thus $\mathbf{n}(\alpha/\lambda)+\coinv(\alpha)$ is at most $\max(\alpha_1,\alpha_2)$.  Since $\varphi(\alpha)$ is formed by moving $\mathbf{n}(\alpha/\lambda)+\coinv(\alpha)$ from $\sort(\alpha)_1=\max(\alpha_1,\alpha_2)$ to $\sort(\alpha)_2$, we have $\varphi(\alpha)_1\ge \lambda_1$, and so $\varphi(\alpha)$ still contains $\lambda$.

  We now show that $\varphi:\mathrm{Alpha}(n,\lambda,2)\to \mathrm{Alpha}(n,\lambda,2)$ is surjective (and hence bijective).  Let $\beta\in \mathrm{Alpha}(n,\lambda,2)$.  If $\beta_2<\lambda_1$, then $\varphi(\beta)=\beta$.  Otherwise, let $d=\beta_2-\lambda_1$.  
  
  If $d$ is even, say $d=2r$, then set $\alpha=(n-(\lambda_1+r),\lambda_1+r)$.  Notice that the first $\lambda_1$ columns of $\beta$ contain $2\lambda_1$ squares, and there are at least $2r$ squares in the remaining columns, so $n\ge 2\lambda_1+2r$.  Thus $n-\lambda_1-r\ge \lambda_1+r$, and so $\alpha$ is a partition, with $\coinv(\alpha)=0$.  The same inequality also shows that $\alpha\supset \lambda$. Thus $\mathbf{n}(\alpha/\lambda)=r$ and it follows that $\varphi(\alpha)=\beta$.
  
  If $d$ is odd, say $d=2r+1$, then set $\alpha=(\lambda_1+r,n-\lambda_1-r)$.  The same calculation as above shows that $\alpha\supset \lambda$ and $\alpha$ is not a partition, so $\coinv(\alpha)=1$.  Furthermore, we again have $\mathbf{n}(\alpha/\lambda)=r$, so $\varphi(\alpha)=\beta$.
\end{proof}

We now construct a bijection from $\mathcal{A}(n,\lambda,2)$ to $\mathcal{T}^+(n,\lambda,s)$ as follows.

\begin{definition}
Let $(\alpha,U)\in \mathcal{A}(n,\lambda,2)$. Define $\psi(\alpha,U)$ to be the tableau formed by changing $1$'s to $2$'s in the bottom row of $U$, starting with the rightmost $1$ and moving leftwards, until we obtain a tableau of content $\varphi(\alpha)$.
\end{definition}

Continuing our running example with $\alpha = (5,6)$, letting $U$ be the following tableau with $\ch(U)=2$, then $\psi(\alpha, U)$ is as below:
\begin{align*}
U = {\small\begin{ytableau}
      2 & 2 & 2 & \none & \none & \none & \none & \none\\
      1 & 1 & 1 & 1 & 1 & 1 & 2 & 2
    \end{ytableau}}\,,   \hspace{1cm} 
\psi(\alpha,U) = {\small\begin{ytableau}
      2 & 2 & 2 & \none & \none & \none & \none & \none\\
      1 & 1 & 1 & 2 & 2 & 2 & 2 & 2
    \end{ytableau}}
\end{align*}

\begin{remark}
The tableau $\psi(\alpha,U)$ is not necessarily semistandard; it may have columns containing two $2$'s.
\end{remark}

\begin{definition}\label{def:Phi}
 Let $(\alpha,U)\in \mathcal{A}(n,\lambda,2)$.  Define $\Phi(\alpha,U)$ as follows.  First, compute $\psi(\alpha,U)$, and append $1$'s to the left of the bottom row and $2$'s to the left of the top row until the resulting tableau $S$ has content $\Lambda$ (and then left-justifying).  Then, unbump a horizontal strip of size $n-k$ from $S$ from right to left to form a tableau $T$ of the same shape as $U$, and an unbumped row of length $n-k$ that acts as the battery of $T$.  We set $\Phi(\alpha,U)=T$.
 \end{definition}

  For our running example, we have $\Lambda_{n,\la,s} = (10,8)$ and
\[
\Phi(\alpha,U) = {\small\begin{ytableau}
      2 & 2 & 2 \\
      1 & 1 & 1 & 1 & 1 & 1 & 1 & 1\\
      \none & \none & \none & \none & \none & \none & \none & \none & 1 & 1 & 2 & 2 & 2 & 2 & 2
    \end{ytableau}}
\]
so that $\ch(\Phi(\alpha,U)) = 5 = \ch(U) + \mathbf{n}(\alpha/\la)+\coinv(\alpha)$. 

\begin{lemma}
    The tableau $T=\Phi(\alpha,U)$ is always well defined and in $\mathcal{T}^+(n,\lambda,2)$.
\end{lemma}

\begin{proof}
  We first note that the intermediate tableau $S$ in Definition \ref{def:Phi} is semistandard, even though $\psi(\alpha,U)$ does not have to be; since $S$ has partition content $\Lambda$ and all of the $1$'s are in the bottom row, this follows immediately.  
  Now, since the shape of $S$ contains the shape of $U$, we can unbump the appropriate horizontal strip from right to left to form $T$.  The resulting letters that were bumped out are in weakly decreasing order from right to left, and therefore form a valid $1\times (n-k)$ battery for $T$.  Finally, since $S$ has content $\Lambda$ by default, the conclusion follows.
\end{proof}

\begin{lemma}
    If $T=\Phi(\alpha,U)$ then $\ch(T)=\ch(U)+\mathbf{n}(\alpha/\lambda)+\coinv(\alpha)$.
\end{lemma}

\begin{proof}
  Note that $\ch(U)$ is the number of $2$'s on the bottom row of $U$.  Therefore, the charge of the tableau $S$ formed from $U$ in Definition \ref{def:Phi} is equal to $$\ch(S)=\ch(U)+\mathbf{n}(\alpha/\lambda)+\coinv(\alpha)$$ since this is the total number of $2$'s on the bottom row.  When we unbump, the charge of the tableau $T$ union with the battery is the same as $\ch(S)$ since charge is invariant under Knuth equivalence.  Thus $\ch(T)=\ch(S)$ and the conclusion follows.
\end{proof}

\begin{thm}
    The map $\Phi$ is a bijection from $\mathcal{A}(n,\lambda,2)$ to $\mathcal{T}^+(n,\lambda,2)$.
\end{thm}

\begin{proof}
  We reverse $\Phi$ as follows.  Given a tableau $T\in \mathcal{T}^+(n,\lambda,2)$, insert its battery to form a tableau $S$.  Then remove $1$'s from the bottom row and $2$'s from the top row so that the remaining letters in each row, when left justified, forms a (not necessarily standard) tableau $U'$ of shape $\sh^+(T)$.  
  Now, if $\beta$ is the content of $U'$, we change $2$'s to $1$'s in the bottom row to form a tableau $U$ of content $\alpha=\varphi^{-1}(\beta)$.  The pair $(\alpha,U)$ is our output.
  
  Once we show that this process is well defined, it is clear that it reverses each step of $\Phi$.  The insertion process to form $S$ is known to be well defined.  For the next step, to show there are enough $1$'s and $2$'s to remove from $S$ to form a tableau $U'$ of shape $\sh^+(T)$, certainly the top row is long enough since it is at least as long as the top row of $T$.  For the bottom row, since the battery that we inserted had length $n-k$, we have to remove at most $n-k$ squares containing $1$ from $S$, and since $\Lambda=(n-k+\lambda_1,n-k+\lambda_2)$, there are at least $n-k$ such squares.
  
  For the last step, by Proposition \ref{prop:alpha-2-row} it suffices to show that $\beta\in \mathrm{Alpha}(n,\lambda,2)$, that is, that the composition $\beta$ contains $\lambda$.  Since there are $n-k+\lambda_1$ squares labeled $1$ in $S$ and we remove at most $n-k$ of them to form $U'$, we have that $\beta_1$, the number of $1$'s in $U'$, is at least $\lambda_1$.  Similarly $\beta_2\ge \lambda_2$, and we are done.
\end{proof}

\section{The \texorpdfstring{$s_{(n)}$}{} coefficient in the \texorpdfstring{$R_{n,k}$}{} case}\label{sec:Rnk}

We now consider the setting in which $\lambda=(1^k)$ and $s=k$, so that $R_{n,\la,s}=R_{n,k}$ and $\widetilde{H}_{n,\lambda,s} = \grFrob(R_{n,k})$, and give a direct combinatorial proof of Theorem \ref{conj:main} for the coefficient of $s_{(n)}$ in this setting.  We recall the positive Schur expansion of $\grFrob(R_{n,k})$ given in \cite{7authors}.  An \textbf{ordered set partition}, or OSP, of $n$ is a partition of $\{1,2,\ldots,n\}$ into a disjoint union of subsets called \textbf{blocks}, along with an ordering of the blocks from left to right.  For instance, $(45|367|28|19)$ denotes an OSP of $9$.  

A \textbf{descent} of a permutation $\pi$ is an index $d$ such that $\pi_d>\pi_{d+1}$, and the \textbf{major index} of $\pi$ is the sum of its descents.  The \textbf{minimaj} of an OSP, first introduced in the context of the Delta conjecture in \cite{HRW}, is the major index of the \textbf{minimaj word} formed by ordering each block's entries from least to greatest and then reading the letters in the OSP from left to right.  For instance, the associated word to $(45|367|28|19)$ is $453672819$, and it has descents in positions $2,5,7$, so the minimaj is $2+5+7=14$.

The \textbf{reading word} $\mathrm{rw}(P)$ of an OSP $P$ (different from its minimaj word) is formed by reading the smallest entry of each block from right to left, and then the remaining entries from left to right.  For instance, the reading word of $(45|367|28|19)$ is $123456789$.

It was shown in \cite{HRS} (using the work of \cite{HRW}) that there is a more general set of \textit{ordered multiset partitions} into $k$ blocks, $\mathcal{OP}_{n,k}$, and a minimaj statistic on them such that $$\omega\circ\revq( \grFrob R_{n,k})=\sum_{\pi\in \mathcal{OP}_{n,k}}q^{\mathrm{minimaj}(\pi)}x^{\mathrm{wt}(\pi)}$$ where $\mathrm{wt}(\pi)$ is the tuple whose $i$-th term is the number of $i$'s in $\pi$.  In \cite{7authors}, a crystal structure is given on ordered multiset partitions that is compatible with the minimaj statistic, thereby grouping the terms of the above monomial expansion into a Schur expansion:
$$\revq( \grFrob R_{n,k})=\omega\circ \sum_{\substack{\pi\in \mathcal{OP}_{n,k}\\ \widetilde{e}_i(\pi)=0 \,\,\forall i}}q^{\mathrm{minimaj}(\pi)}s_{\mathrm{wt}(\pi)}=\sum_{\substack{\pi\in \mathcal{OP}_{n,k}\\ \widetilde{e}_i(\pi)=0 \,\,\forall i}}q^{\mathrm{minimaj}(\pi)}s_{\mathrm{wt}(\pi)^\ast},$$ where $\widetilde{e}_i$ are the raising operators of the crystal, which we define below. 

In particular, the coefficient of $s_{(n)}$ in the above expansion (taking into account the conjugation via $\omega$) is equal to $$    \sum_{\substack{P\in \mathcal{OP}_{n,k}, \mathrm{wt}(P)=(1^n) \\ \widetilde{e}_i(P)=0 \,\,\forall i}} q^{\mathrm{minimaj}(P)}=\sum_{\substack{P\in \mathrm{OSP}(n,k) \\ \widetilde{e}_i(P)=0 \,\,\forall i}} q^{\mathrm{minimaj}(P)}
$$ where $\mathrm{OSP}(n,k)$ is the set of ordered set partitions with entries $1,2,\ldots,n$ and $k$ blocks.  

The crystal raising operators $\widetilde{e}_i$ were defined in \cite{7authors} via the reading word described above.  In particular, $\widetilde{e}_i(P)=0$ if and only if, in the reading word, the number of $i$'s is always greater than or equal to the number of $i+1$'s as we read the word from left to right.  Thus if $P$ has content $(1^n)$, we have $\widetilde{e}_i(P)=0$ for all $i$ if and only if the reading word of $P$ is $123\cdots n$.  Thus the coefficient of $s_{(n)}$ in $\revq(\grFrob(R_{n,k}))$ is
\begin{equation}\label{eq:minmaj}
    \sum_{\substack{P\in \mathrm{OSP}(n,k) \\ \mathrm{rw}(P)=123\cdots n}} q^{\mathrm{minimaj}(P)}.
\end{equation}
On the other hand, the coefficient of $s_{(n)}$ in the charge formula of Theorem \ref{conj:charge} is 
\begin{equation}\label{eq:ch-formula}
\sum_{\substack{T\in \mathcal{T}^+(n,(1^k),k) \\ \sh^+(T)=(n)}} q^{\ch(T)}.
\end{equation}
To prove that \eqref{eq:minmaj} and \eqref{eq:ch-formula} are equal via combinatorial methods, we first prove a lemma about charge, and then we define a bijection $f$ from the set of tableaux $T$ appearing in the sum \eqref{eq:ch-formula} to the OSPs in \eqref{eq:minmaj} as follows.

\begin{lemma}\label{lem:chRnk}
    Given $T\in\mathcal{T}^+(n,(1^k),k)$ such that $\sh^+(T) = (n)$, the charge labels of the battery of $T$ are always either $0$ or $1$, with the $1$ labels being precisely on the entries of the battery that are larger than their row index. Furthermore, all charge labels in the device are $0$ except in the final charge word which is $123\cdots k$ in order.
\end{lemma}

\begin{proof}
We proceed by induction on $n-k$. In the base case when $n-k=0$, the battery is empty, so $T$ has content $\Lambda = (1^n)$ in this case, so there is only one charge word which consists of the entire row labeled $12\cdots n$ in order (where $n=k$), so the base case holds.

Letting $n-k>0$ and $T\in \mathcal{T}^+(n,(1^k),k)$ such that $\sh^+(T) = (n)$, let $i$ be minimal such that $i$ does not appear in row $i$ of the battery, or $i=k$ if such an $i$ does not exist. Then since $\sh^+(T) = (n)$, the first charge word of $T$ consists of the last $j$ entry of row $j$ of the battery for each $j < i$, together with the right-most $i$ in the device, and the right-most $j$ of the battery in row $j-1$ for $i<j\leq k$. Thus, the charge labels for $j\leq i$ are $0$ and for $j>i$ they are all $1$. 

Deleting $i$ from the device and left-justifying, and deleting the other entries of the first charge word from the battery and left justifying each row of the battery, we get a battery-powered tableau $T'\in \mathcal{T}^+(n-1,(1^k),k)$ with $\sh^+(T') = (n-1)$. The charge labels for the entries of $T'$ are the same as the charge labels of the corresponding cells of $T$. By our inductive hypothesis, we are done.
\end{proof}

\begin{figure}
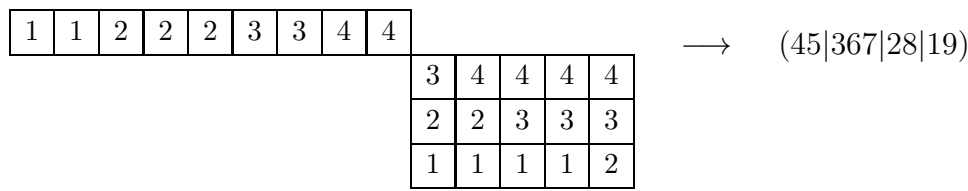

    \centering
$${\small\begin{ytableau}
      1 & 1 & 2 & 2 & 2 & 3 & 3 & 4 & 4 \\
      \none & \none & \none & \none & \none & \none & \none & \none & \none & 3 & 4 & 4 & 4 & 4  \\
      \none & \none & \none & \none & \none & \none & \none & \none & \none & 2 & 2 & 3 & 3 & 3  \\
      \none & \none & \none & \none & \none & \none & \none & \none & \none & 1 & 1 & 1 & 1 & 2  \\
    \end{ytableau}}\hspace{0.5cm}
    \longrightarrow \hspace{0.5cm}
    (45|367|28|19)
$$
    \caption{A battery-powered tableau $T$ of shape $(9)$ for $\la=(1^4)$ and $s=4$, and the corresponding ordered set partition $P$.  We have $\ch(T)=\mathrm{minimaj}(P)=14$.}
    \label{fig:Rnk}
\end{figure}

\begin{definition}\label{def:f}
 Given $T\in \mathcal{T}^+(n,(1^k),k)$ with shape $(n)$, define $f(T)$ to be the ordered set partition constructed as follows.  Let $f(T)$ have exactly $k$ blocks $B_1,\ldots,B_k$ in that order, which initially contain $k,k-1,k-2,\ldots,1$ respectively.   Then let $m_i$ be the number of $i$'s in the device of $T$, and place the numbers $k+1,k+2,\ldots,n$ into the blocks from left to right in the unique way so that each block $B_i$ has size $m_i$ for all $i$.  The resulting OSP is $f(T)$.
\end{definition}

An example of $f(T)$ is depicted in Figure \ref{fig:Rnk}.

\begin{prop}
    The assignment $T\mapsto f(T)$ is a bijection from the set of all tableaux $T\in \mathcal{T}^+(n,(1^k),k)$ such that $\sh^+(T) = (n)$ to the set of $P\in \mathrm{OSP}(n)$ such that $\mathrm{rw}(P) = 123\cdots n$. The map $f$ is weight preserving, meaning that $\ch(T) = \mathrm{minimaj}(f(T))$.
\end{prop}

\begin{proof}
To show $f$ is well defined, observe that $\Lambda_{n,(1^k),k}=((n-k+1)^k)$, and so $T$ has exactly $n-k+1$ copies of each letter from $1$ through $k$. Since the battery of $T$ has $n-k$ columns, then there must be at least one of each $i\leq n$ in the device of $T$.  In the notation of Definition~\ref{def:f}, we thus have $m_i\ge 1$ for all $i$, so $f(T)$ is a well-defined OSP.  By its construction, the reading word of $f(T)$ is $123\cdots n$, and the process is reversible since there is a unique way to fill the one-row device and the battery for any sequence of block sizes $m_i$.  Thus $f$ is a bijection. 

We now prove that $f$ is weight-preserving, sending $\ch$ to $\mathrm{minimaj}$.  Indeed, by Lemma~\ref{lem:chRnk}, the final charge subword, which is $123\cdots k$ in order, has charge $\binom{k}{2}$.  This is the minimaj value formed by placing $k,k-1,\ldots,1$ in the blocks from left to right. For each $i$ in the device of $T$ that is not in the final charge subword, the charge labels of the $i+1,\dots, k$ in the charge subword of $i$ are all $1$, so $i$ contributes $k-i$ to charge. In terms of minimaj, adding an extra element to $B_i$ increases the minimaj corresponding to blocks $B_{i+1},\dots, B_k$, and thus results in an increase of $k-i$. Thus, placing the remaining letters in the blocks increases the minimaj by precisely the amount of charge stored in the battery.
\end{proof}

\section{Skewing formulas for the Delta Conjecture at low \texorpdfstring{$t$}{} degrees}\label{sec:conjs}

It is natural to ask whether our skewing formula for the Delta Conjecture at $t=0$ extends to the full Delta Conjecture symmetric function. In this section, we give several conjectures of such expansions below. Each formula may be expanded in order to obtain a positive Schur expansion.

\begin{example} In the case $n=4$, $k=3$, the skewing formula generalizes to the following skewing formula for the full Delta Conjecture symmetric function.
\begin{align*}
\omega\Delta'_{e_2} e_4 &= s_{(1,1)}^\perp \bigg(H_{(2,2,2)}(x;q) + (t(1+q)+t^2)H_{(3,2,1)}(x;q)\\
&\,\,\,\,\,\,+ (t^2(1+q) + 2t^3 
+ t^4)H_{(4,2)}(x;q) + (t^3 + t^4 + t^5)H_{(5,1)}(x;q)\bigg),
\end{align*}
which in turn gives a Schur-positive expansion for $\Delta'_{e_2}e_4$ after expanding each Hall-Littlewood polynomial in terms of charge.

Similarly, for $n=5$ and $k=3$, we have
\begin{align*}
    \omega\Delta'_{e_{2}}e_5 &= s_{(2,2)}^\perp\bigg( H_{(3,3,3)} + t(1+q)H_{(4,3,2)} + (t^2(1+q) + t^3 + t^4)H_{(5,3,1)} + t^3H_{(4,4,1)}\\
    &+ (t^3+t^4+t^5)H_{(5,4)} + t^3H_{(6,2,1)} + (t^3+2t^4+2t^5+t^6)H_{(6,3)}  \\
    &+(t^4+2t^5+t^6+t^7)H_{(7,2)}\bigg).
\end{align*}
Alternatively, the terms $t^3(H_{(6,2,1)} + H_{6,3})$ may be replaced with $t^3((q+2)H_{(6,3)} + H_{(7,2)})$.
\end{example}

\begin{remark}
    In general, $\omega\Delta'_{e_{k-1}} e_n$ does not have an expansion as a single $s_\lambda^\perp$ applied to a positive sum of Hall-Littlewood polynomials - for instance, the $t^4$ coefficient in $\omega \Delta'_{e_{3}}e_5$ is not Hall-Littlewood positive (and it is known that $s_\la^\perp$ applied to a Hall-Littlewood polynomial is Hall-Littlewood positive).  That being said, in the conjectures below we find some formulas of this form for the coefficients of low-degree powers of $t$. 
\end{remark}

Let $[n]_q = 1+q + \cdots + q^{n-1}$ and $\binom{n}{m}_q = \frac{[n]_q}{[m]_q[n-m]_q}$ be the usual $q$-analogues.

\begin{conj}\label{conj:t-coeff}
    The coefficient of $t^1$ in $\omega\Delta'_{e_{k-1}}e_n$ (as a polynomial in $t$ with coefficients in symmetric functions over $\bQ[q]$) is
    \[
    [k-1]_q \cdot s_{((n-k)^{k-1})}^\perp H_{(n-k+2,(n-k+1)^{k-2},n-k)}(x;q)
    \]
    where $(n-k+2,(n-k+1)^{k-2},n-k)$ is shorthand for the partition $(n-k+2,n-k+1,n-k+1,\ldots n-k+1,n-k)$ with $k-2$ copies of the part $n-k+1$.
\end{conj}

\begin{conj}\label{conj:t-squared}
    The $t^2$ coefficient of $\omega\Delta'_{e_{k-1}}e_n$ is
    \begin{gather*}
    s_{((n-k)^{k-1})}^\perp \bigg( [k-2]_q H_{(n-k+2,(n-k+1)^{k-2},n-k)}(x;q) 
    + \binom{k-2}{2}_q H_{((n-k+2)^2,(n-k+1)^{k-4},(n-k)^2)}(x;q)\\
    + [k-1]_q H_{(n-k+3,(n-k+1)^{k-2},n-k-1)}(x;q)\bigg).
    \end{gather*}
\end{conj}

We have checked both Conjectures~\ref{conj:t-coeff} and \ref{conj:t-squared} computationally up to $n=8$ for all $k\leq n$.
In the case of $k=2$, we have the following formula for the full Delta Conjecture symmetric function.

Before we state the formula, we recall the Littlewood-Richardson rule for skew Schur functions in the case of two-row partitions. Given $\la = (\la_1,\la_2)$ and $\mu = (\mu_1,\mu_2)$ partitions,
\[
s_\mu^\perp s_\la = s_{\la/\mu} = \sum_{\nu\vdash |\la/\mu|} c_{\mu,\nu}^\la s_\nu(x)
\]
where $c_{\mu,\nu}^\la$ is the number of semistandard Young tableaux $T$ of skew shape $\la/\mu$ with content $\nu$ whose reverse reading word is \textbf{Yamanouchi}, meaning that if one reads the labels of $T$ in reverse reading order, there are never more $2$s than $1$s up to any given point.

\begin{prop}
    For $k=2$, we have
    \begin{equation}\label{eq:DeltaFor2}
    \omega\Delta'_{e_{1}}e_n = h_{n-2}^\perp \sum_{i=0}^{n-1} H_{(n-1+i,n-1-i)}(x;q)t^i.
    \end{equation}
\end{prop}

\begin{proof}
    By~\cite[Proposition 6.1]{HRW}, 
    \begin{equation}\label{eq:HRW_formula}
    \omega\Delta'_{e_{1}}e_n = \omega\Delta_{e_1}e_n - \omega e_n = -s_{(n)} + \sum_{i=0}^{\lfloor n/2\rfloor} s_{(n-i,i)}(x)\sum_{p=i}^{n-i} [p]_{q,t},
    \end{equation}
    where $[p]_{q,t} = \sum_{j=0}^{p-1} q^j t^{n-1-j}$.

    Notice that the Hall-Littlewood term $H_{(n-1+i,n-1-i)}(x;q)$ on the right hand side of \eqref{eq:DeltaFor2} expands as $\sum_{j=0}^{n-1-i} q^{j} s_{n-1+i+j,n-1-i-j}$ in the Schur basis, by examining the charge expansion version of \eqref{eq:lascoux-schutzenberger} in this two-row case.  Starting from the right-hand side of \eqref{eq:DeltaFor2}, 
    \begin{align}
        h_{n-2}^\perp \sum_{i=0}^{n-1} H_{(n-1+i,n-1-i)}(x;q)t^i&= h_{n-2}^\perp \left(\sum_{i=0}^{n-1} \sum_{j=0}^{n-1-i} t^iq^j s_{(n-1+i+j,n-1-i-j)}\right)\\
        &=s_{n-2}^\perp\left(\sum_{\ell=0}^{n-1} s_{(n-1+\ell,n-1-\ell)}(q^\ell+q^{\ell-1}t+\cdots+t^\ell)\right) \\
        &= s_{n-2}^\perp\left(\sum_{\ell=0}^{n-1} s_{(n-1+\ell,n-1-\ell)}[\ell+1]_{q,t}\right). \label{eq:before-skew}
    \end{align}
We now examine the coefficient of $s_{(n-i,i)}$ in \eqref{eq:before-skew}.  Applying the Littlewood-Richardson rule to compute $s_{n-2}^\perp s_{(n-1+\ell,n-1-\ell)}$ over all $\ell$, we have that $[\ell+1]_q$ appears once in the coefficient of $s_{(n-i,i)}$ if and only if there exists a Littlewood-Richardson tableau of skew shape $(n-1+\ell,n-1-\ell)/(n-2)$ and content $(n-i,i)$ (and note that since we are in the two-row case, there can only be one such tableau if it exists).

There are two inequalities that govern the existence of such a tableau in the case when $i\ge 1$ and hence there is at least one $2$.  First, the number of $2$s cannot exceed the number of entries in the bottom row (which must all be $1$) by the Yamanouchi condition, so we have $i\le (n-1+\ell)-(n-2)=\ell+1$.  Second, the number of $2$s naturally cannot exceed the size of the top row, and so $i\le n-1-\ell$.  Solving these two inequalities for $\ell$, we find $i-1\le \ell\le n-1-i$.  Finally, all such fillings are semistandard, since the only shape that has a vertical domino is when $\ell=0$, and it has a unique vertical domino, so having at least one $2$ (due to our assumption that $i>1$ in this case) guarantees the existence of the desired Littlewood-Richardson tableau.  It follows that the coefficient of $s_{(n-i,i)}$ is equal to $\sum_{\ell=i-1}^{n-1-i}[\ell+1]_{q,t}=\sum_{p=i}^{n-i}[p]_{q,t}$.  

Finally, we examine the coefficient of $s_{(n)}$.  The same analysis as above goes through, except in the case that $\ell=0$, when the constructed tableau would not be semistandard.  Thus the coefficient of $s_{(n)}$ is $-1+\sum_{p=0}^{n}[p]_{q,t}$, and we are done.
\end{proof}

\begin{remark}
    Alternatively, all of the formulas in this section may be written as formulas for $\Delta'_{e_{k-1}}e_n$ in terms of $q$-Whittaker polynomials $\omega H_\mu(x;q)$ by applying $\omega$ to both sides and replacing the operator $s_{((n-k)^{k-1})}^\perp$ with $s_{((k-1)^{n-k})}^\perp$.
\end{remark}

\section{Next directions}\label{sec:next}

The new results and connections to geometry in this paper open up several natural directions for further investigation.

\begin{q}
 Are the $\Delta$-Springer varieties the only family of Borho--MacPherson $\sP_x^y$ varieties that have sufficient rational smoothness properties to obtain a simple Schur expansion for the graded Frobenius of their cohomology rings?  If not, which others may lead to useful combinatorial formulas?
\end{q}

This paper rests in type A, but the Borho--MacPherson paper is type independent, so we also ask the following.

\begin{q}
   Is there a natural extension of $\Delta$-Springer varieties to all Lie types that has combinatorial meaning?
\end{q}

On the combinatorics side, since Corollaries \ref{cor:DeltaPerpFormula} and \ref{cor:Rnk} give formulas for the $t=0$ specialization of the Delta Conjecture, and Section~\ref{sec:conjs} gives conjectures for other $t$ degrees, we also ask whether we can extend these formulas to the full Delta Conjecture symmetric functions for all $t$ degrees.

\begin{q}
 Can $\Delta'_{e_{k-1}}e_n$ be obtained by applying a t-analogue of a skewing operator to a Macdonald polynomial, generalizing Corollary \ref{cor:DeltaPerpFormula}? Does Corollary \ref{cor:Rnk} have a $q,t$-analog that gives a Schur expansion or other formula relevant to the Delta Conjecture?
\end{q}

Finally, the proofs in this paper rely heavily on the deep geometric, topological, and representation-theoretic machinery developed by Borho and MacPherson.   We would like to see a combinatorial proof along the lines of the Lascoux--Sch\"utzenberger proof of the Hall-Littlewood cocharge formula (see \cite{Butler} for a modern exposition of this proof).

 \begin{q}
Is there a more direct combinatorial or algebraic proof of Theorem \ref{conj:main}?

In particular, in Section \ref{sec:Rnk}, we used the known Schur expansion of \cite{7authors} for the $R_{n,k}$ case in terms of minimaj to give a second proof that the formula of Theorem \ref{conj:main} holds for the $s_{(n)}$ coefficient. Is there a generalization of the minimaj Schur expansion to the setting of $\widetilde{H}_{n,\lambda,s}$ that would allow us to obtain a combinatorial proof for the $s_{(n)}$ coefficient in the general case?
 \end{q}

The companion paper \cite{preprint} will also investigate combinatorial routes towards Theorem \ref{conj:main} via a new formula in terms of Compositional Shuffle Theorem creation operators \cite{CarlssonMellit,HMZ}.

Combining Theorem~\ref{thm:skew-formula} and \eqref{eq:HL-expansion}, our result gives a formula for the symmetric function $s_{((n-k)^{s-1})}^\perp\widetilde{H}_\Lambda$ as a positive sum of Hall-Littlewood polynomials. Furthermore, by  \cite{GarsiaProcesi} there is also a formula for $e_j^\perp\widetilde{H}_\nu$ for any $j$ and $\nu$ as a sum of Hall-Littlewood polynomials.

\begin{q}
    Is there a combinatorial formula for $s_\mu^\perp \widetilde{H}_\nu$ in terms of Hall-Littlewood polynomials that generalizes the expansion \eqref{eq:HL-expansion} to all $\mu$ and $\nu$?
\end{q}

\printbibliography

\end{document}